\documentclass{amsart}

\usepackage{amsmath, amssymb, enumerate, hyperref, graphicx, setspace, color}
\usepackage[left=1in,top=1in,right=1in,bottom=1in]{geometry}
\usepackage[usenames,dvipsnames]{xcolor}

\usepackage{tikz}
\usetikzlibrary{calc, arrows, shapes}
\usepackage{tikz-cd}

\onehalfspace
\numberwithin{equation}{section}

\theoremstyle{definition}
\newtheorem{thm}{Theorem}
\newtheorem{cor}[thm]{Corollary}
\newtheorem{lem}[thm]{Lemma}
\newtheorem{prop}[thm]{Proposition}

\theoremstyle{definition}
\newtheorem{defn}[thm]{Definition}

\newtheorem{rem}[thm]{Remark}
\newtheorem{ex}[thm]{Example}

\newcommand{\A}{\mathcal A}
\newcommand{\B}{\mathcal B}
\newcommand{\C}{\mathcal C}

\newcommand{\M}{\mathcal M}
\newcommand{\Q}{\mathcal Q}

\newcommand{\T}{\mathcal T}

\newcommand{\R}{\mathbb R}
\newcommand{\N}{\mathbb N}
\newcommand{\Z}{\mathbb Z}

\newcommand{\catname}[1]{{\normalfont\textbf{#1}}}
\newcommand{\catC}{\catname{C}}
\newcommand{\catD}{\catname{D}}
\newcommand{\catJ}{\catname{J}}
\newcommand{\Set}{\catname{Set}}
\newcommand{\SM}{\catname{StatMeas}}

\newcommand{\Top}{\catname{Top}}
\newcommand{\Mon}{\catname{Mon}}
\newcommand{\Meas}{\catname{Meas}}

\renewcommand{\d}{\textbf{d}}
\renewcommand{\r}{\textbf{r}}
\newcommand{\E}{\operatorname{E}}

\newcommand{\Dom}{\operatorname{Dom}}
\renewcommand{\Im}{\operatorname{Im}}

\title{Tarski Measure}
\author{Bryson}
\date{Spring 2013}

\begin{document}

\maketitle
\tableofcontents

\section{Introduction}
This work investigates the intrinsic properties of measurements defined by symmetries of spaces. Its genesis lies in Tarski's measure theoretic work on stationary measures and the attendant duplication theorems.  Our purpose is twofold.  Firstly, we wish to identify natural monoids of \emph{sizes} or \emph{quantities} given a measurable space with some symmetries and elucidate the structure of these monoids. And second, we would like to present a suitable category theoretic framework within which to study the relationships between such spaces with symmetry.  This second goal is a byproduct of the current atmosphere in Mathematics. For some time it has been clear that it is fruitful to study not only objects of interest but the connections between them. Categories formalize this. The first goal stands in sharp contrast with the historical efforts which additionally assume a model for quantities: namely, the non-negative real numbers. While this view has some roots in antiquity, since the introduction of Measure Theory just over a century ago, its been axiomatic.

Henr\'i Lebesgue is the father of modern measure and integration theory.  His celebrated 1902 dissertation \emph{Intégrale, longueur, aire} is foundational.  Within its pages and the lectures following it, Lebesgue introduced $\sigma$-algebras, Lebesgue measure, integration, the dominated convergence theorem, etc. It is difficult to overestimate the importance of Lebesgue's work. It is here that we get the modern definition of a measure as a countably additive map from a $\sigma$-algebra to $[0,\infty]$ taking the empty set to $0$. The dissertation lays out a remarkably successful extension of Calculus which has proved so amenable to further generalizations that Lebesgue's definitions have become ubiquitous.

However, it is the common necessity (particularly in the presence of symmetries) of working with $\sigma$-algebras much smaller than the full power set of a space that first interested me in Measure Theory.  To wit, I am still amazed by Giuseppe Vitali's 1905 construction of a non-Lebesgue measurable set. This construction was the first published demonstration of non-measurability and forcefully makes the point that jointly assuming there exists a map $\mu: 2^\R \to [0,\infty]$ satisfying (i) $\mu$ is countably additive, (ii) $\mu$ is stationary with respect to the action of the Euclidean group, and (iii) $\mu([0,1]) = 1$ must lead to contradictions.

In the two decades following Vitalli's discovery a pair of counter-intuitive results were proven. They more clearly demonstrated the underlying tension between Lebsegue's model for measure and symmetries. 1914 brought the Hausdorff paradox which shows that one can nearly duplicate the 2-sphere using a finite partition and transformations from $SO_3$. Then, the 1924 proof of Banach and Tarski's infamous theorem on the duplication of spheres: It is possible to partition a solid ball in $\R^3$ into finitely many pieces such that the pieces can be rearranged by rotation and translation into two solid balls of the same size as the original. The common thread binding these `paradoxical' results is the ingenious use of geometric constructions to reveal inconsistencies in the mathematical formalization of stationary measures.

By 1938 Tarski had worked out a rich theory on the existence (or non-existence) of finitely additive stationary measures in great generality. Tarski realized that he could weaken the assumptions on his symmetries from being groups to pseudogroups, and indeed inverse semigroups which allows one to model local symmetries as well as global symmetries.\footnote{Lawson's text \emph{Inverse Semigroups}~\cite{lawson1998} covers such notions comprehensively.  In differential geometry, \emph{local structures} provide a natural example of geometries for which there is not necessarily a group but naturally an inverse semigroup of symmetries.  See Appendix~\ref{sym}.}  His work provides necessary and sufficient algebraic conditions for a set with an inverse semigroup of symmetries to carry a stationary measure defined on all subsets of the set. This characterization turns critically on whether or not there exist subsets which may be duplicated by the actions of the symmetries and of partitions. Tarski's formalism and its countably additive generalizations provide a foundation from which we may investigate abstract measures taking values in monoids other than $\overline{\R}^+$.

The central philosophy of measurement present in Tarski's work holds that within a space with some symmetry there is a natural notion of `same size'.  Two objects which are similar are of the same size. Two objects which can be partitioned into pieces which are pairwise similar are of the same size. This notion is taken as axiomatic in Tarski's measure theory but it is informed by common assumptions about size. Imagine a cake at a party.  Consider the process of cutting the cake into pieces and placing those pieces at each setting on a table.  It is reasonable to claim that the original cake and together the collection of pieces of cake, are of the same size.  Why? Because we assume that size is preserved by cutting up (partitioning) and by moving the parts (the action of symmetries).

Tarski's essential departure from the classical measure theory of Lebesgue was to emphasize symmetry as a defining feature of measure.  Often bundled together are the observation that, whatever size is, it should be preserved by partition and symmetry, AND the assumption that sizes are modelled by non-negative real numbers: ``I have 29 cubic inches of cake!"  Tarski parsed this out. He defined the \emph{type} of a measurable object to be the equivalence class of objects which---according to partition and symmetry---were necessarily the same size.  In so doing he constructed the \emph{type monoid} of a stationarily measurable space.  Then, and only then, did he attempt to represent types by non-negative real numbers. 

It is astonishing how much structure is present in these type monoids. They are canonically partially ordered commutative monoids whose addition law is subtly intertwined with the idempotent elements. These monoids embed into inverse semigroups called \emph{quantity spaces} in an analogous fashion to the embedding of $[0,\infty]$ into the one point compactification of the reals $\R \sqcup \{ \infty \} $. Moreover, this embedding gives rise to a hierarchical view of measurement that restricts to various `scales' modelled by the isotropy groups of quantity spaces. Perhaps most importantly, the map assigning to each measurable set, the type of that set is an intrinsic measure. This \emph{Tarski measure} satisfies in great generality the basic theorems of measure theory. 

In particular, this work introduces the category $\SM$ of stationarily measurable spaces together with the construction of a co-functor $\T: \SM \to \Mon$.\footnote{Where $\Mon$ is the category of monoids in $\Set$ consisting of monoids in the usual sense---as semigroups with unit---and their homomorphisms.}
The construction induces a map referred to as \emph{Tarski Measure} having many of the properties of a stationary measure as summarized in Section~\ref{results}.  Additionally, we identify the lattice of minimal $\sigma$-ideals of null sets for stationary measures which do not witness paradoxes of measure. Throughout, we make use of several elementary results from the theory of Inverse Semigroups (which we use to model symmetries) as well as some basic categorical concepts (functors, coproducts, limits, etc.) where they simplify the expression of relationships.  Appendices~\ref{sym} and~\ref{cats} present this material with some examples to aid the reader unfamiliar with these topics. 

\section{Stationarily Measurable Spaces}
The category of stationarily measurable spaces $\SM$ consists of measurable spaces with a notion of symmetry modelled by an inverse monoid action. That is, $\SM$ has as objects, triples $(X,\Phi_S,\A)$.  In this triple, $(X,\A)$ is a measurable space e.g. a set $X$ equipped with a $\sigma$-algebra $\A$ of subsets of $X$.  $S$ is an inverse monoid which are a particular generalization of groups in which for each $s \in S$ there is a unique \emph{weak inverse} $s^*$ such that $ss^*s = s$ and $s^*ss^* = s^*$.  Finally, $\Phi_S$ is an action of $S$ on $(X,\A)$.  Viewing $S$ as a category with a single object, this can be presented as a functor $\Phi_S: S \to \Meas$ such that $obj(S)=\{\cdot \}~\overset{\Phi_S}{\longmapsto}~(X,\A)$.  Alternatively, the action can be presented as a map $\Phi_S: S \times X \to X$ such that for each $s \in S$ the map $\Phi_s: X \to X$ is measurable. To simplify notation, we employ the convention of denoting the action as left multiplication whenever possible. Hence, $\Phi_s$ becomes $s$ and $(X,\Phi_S,\A)$ becomes $(X,S,\A)$.

Morphisms in $\SM$ are maps $f:(X,S,\A) \to (Y,T,\B)$ where $f$ is an $(\A,\B)$-measurable function and there exists an inverse monoid homomorphism $f^*:T \to S$ such that for all $t \in T$ and $B \in \B$ we have:
$$
f^{-1}(t^{-1}(B)) = (f^*(t))^{-1}(f^{-1}(B))
$$
or, rewriting the compositions $(t \circ f)^{-1}(B) = (f \circ f^*(t))^{-1}(B)$.  That is, $T$-orbits pullback homomorphically to $S$-orbits.  Throughout, we will consider such triples $(X,S,\A)$ referred to as stationarily measurable spaces.  Stationarily measurable spaces are not uncommon. Topological groups, local structures and many manifolds including $\R^n$ (together with the Euclidean group of transformations) are examples.

It should be noted that despite the strong symmetry condition on the morphisms of $\SM$, in a sense $\SM$ refines the familiar category of measurable spaces, $\Meas$. Recall, $\Meas$ is composed of the measurable spaces which are pairs $(X,\A)$ where $X$ is a set and $\A$ is a $\sigma$-algebra of subsets together with the measurable maps between such spaces. There is an evident forgetful functor $U: \SM \to \Meas$ defined by mapping $(X,S,\A) \mapsto (X,\A)$ and $f \mapsto f$.  Indeed, $\Meas$ is equivalent to the subcategory of spaces in $\SM$ equipped with trivial group actions. To summarize, we note the following result:
\begin{quotation}
The functor $R: \Meas \to \SM$ mapping $(X,\A) \mapsto (X,0,\A)$ and $f \longmapsto f$ assigning the trivial group of symmetries to measurable spaces is right adjoint to the forgetful functor $U:~(X,S,\A) \to (X,\A)$.
\end{quotation}
This result is a formalization of one of the two basic premises behind Tarski measure.  Even if they are trivial, symmetries are always present. The second premise being that symmetries induce measures.  These two ideas are the primary motivation behind the definition of the category $\SM$ as a suitably general context in which measurable spaces equipped with symmetries interact.

\section{Tarski Measure}\label{results}
Tarski~\cite{tarski1938} originated the notion of type monoids for finitely additive measures in 1938 proving two seminal results in the theory of finitely additive measures along the way.\footnote{Stan Wagon's text \emph{The Banach-Tarski Paradox}~\cite{wagon1993} is an excellent reference on theory of finitely additive measures.}  First he gave an algebraic condition for a commutative monoid equipped with its canonical preorder ($\alpha \preceq \beta \iff \exists \gamma$ such that $\alpha + \gamma = \beta$) to admit a representation in the non-negative reals normalizing a specified element. Then he proved an existence theorem for finitely additive stationary measures defined for all subsets of a given set.

More precisely, Tarski presented the following two results.
\begin{quotation}
	If $(M,+,0; \varepsilon)$ is a commutative monoid with specified element $\epsilon$ then: $n \in \N$, $(n+1)\varepsilon \not\preceq n\varepsilon$ for all $n \in \N$ if and only if there is a homomorphism $\mu: M \to [0,\infty]$ such that $\mu(\varepsilon) = 1$
\end{quotation}
and, utilizing the (finitary) type monoid that;
\begin{quotation}
	Let $S$ be an inverse monoid acting on a set $X$ and $E \subseteq X$. Then there exists a finitely additive, $S$-stationary measure $\mu: 2^X \to [0,\infty]$ with $\mu(E) = 1$ if and only if $E$ is not $S$-paradoxical.\footnote{A set $E \subseteq X$ is finitely $S$\emph{-paradoxical} if in the (finitary) type monoid $2[E] \preceq [E]$.}
\end{quotation}
In this manner, Tarski transformed the question of the existence of a stationary measure into a question of the existence of a monoid homomorphism from a (finitary) Type monoid into $\overline{\R}^+$ which he was able to resolve.  Implicit in these results is the idea that (finitely additive) classical stationary measures factor through (finitary) Type monoids. Indeed, Theorem~\ref{meascoor} can be seen as a countably additive analogue to this fact.

In later work, most notably the 1949 text~\emph{Cardinal Algebras}~\cite{tarski1949}, Tarski expanded the formalism to include countably infinitary operations and hence investigation into countably additive stationary measures. Contemporaneously, we have Robinson's 1947 paper \emph{On the Decomposition of Spheres}~\cite{robinson1947} which proved the tight Banach-Tarski paradox: The unit ball in $\R^3$ admits a 5 piece partition (the fifth taken to be a singleton) such that the members of the partition are pairwise congruent to a five piece partition of TWO unit balls and no fewer pieces will suffice. Between 1942 and 1953, Maharam produced a series of papers~\cite{maharam1942, maharam1949, maharam1953} exploring primarily the representation of type monoids and related constructions by real numbers. Chuaqui~\cite{chuaqui1969, chuaqui1977} worked similarly on real number representations of type monoids and more generally Tarski's generalized cardinal algebras.

\subsection{The Tarski Type co-Functor}
While authors differ in approach and formalism, the central aspects of the theory are consistent. In our language these core ideas are as follows. For any stationarily measurable space $(X,S,\A)$ there is an intrinsic notion of \emph{type}. The map, called \emph{Tarski measure} taking a set to its type satisfies many of the familiar properties of an $S$-stationary measure on $\A$. And, under some circumstances the collection of types admits a representation in $[0,\infty]$. This last situation being equivalent to the existence of a (not necessarily unique) classical stationary measure.

The present work's motivation lies in an attempt to gather results in this area and modernize their statement using the categorical framework given here.  The main theorems of this paper, stated below, are a first step in that direction.

\begin{thm}\label{bigtheorem}
	The Tarski Type construction $\T: \SM \to \Mon$ is a co-functor with the following properties:
	
	\begin{enumerate}
		\item \textbf{Properties of the Object Map:} Let $(X,S,\A) \in obj(\SM)$ and $\alpha, \beta \in \T(X,S,\A)$;
			\begin{description}
				\item[Commutativity:]
					$\alpha + \beta = \beta + \alpha$
				\vspace{4 pt}
				\item[Ordering:]
					The canonical preorder ($\alpha \preceq \beta \iff \exists \gamma$ such that $\alpha + \gamma = \beta$) is a partial ordering.
				\vspace{4 pt}
				\item[$\N$-Cancellation:]
					If for some $n \in \N$, $n \alpha = n \beta$; then $\alpha = \beta$.
				\vspace{4 pt}
			\end{description}
		\item \textbf{Properties of Tarski Measure:} There is a canonical mapping $m_{(X,S,\A)}: \A \to \T(X,S,\A)$ referred to as \emph{Tarski Measure}, with the following properties;
			\begin{description}
				\item[Null Identity:]
					$m(\varnothing) = 0$ the additive identity in $\T(X,S,\A)$ 
				\item[Countable Additivity:]
					If $\{A_{i}\}^{\infty}_{i=1}$ is a sequence of disjoint sets in $\A$ then $m\left( \bigcup\limits_{i=1}^{\infty} A_{i} \right) = \sum\limits_{i=1}^{\infty} m(A_{i})$
				\item[Stationarity:]
					$m(A) = m(s^{-1}A)$ for all $s \in S$
			\end{description}
		\item \textbf{Properties of the Morphism Map:} For $f: (X,S,\A) \to (Y,T,\B) \in mor(\SM)$, we have $\T f: \T(Y,T,\B) \to \T(X,S,\A)$ and;
			\begin{description}\label{propmorph}
				\item[$\sigma$-Homomorphic:]
					$\T f \left( \sum\limits_{i=1}^{\infty} a_i \right) = \sum\limits_{i=1}^{\infty} \T f (a_i)$
				\vspace{4 pt}
				\item[Commutativity of Measurement:]
					$m_{(X,S,\A)} \circ f^{-1} = \T f \circ m_{(Y,T,\B)}: \B \to \T(X,S,\A)$
			\end{description}
	\end{enumerate}
\end{thm} 

Type monoids $\T = \T(X,S,\A)$ are universal constructions which play an analogous role to that of the non-negative reals in classical measure theory. Just as $\overline{\R}^+ = [0, \infty]$ embeds into the one point compactification of $\R$ to great effect---for instance, allowing the proof of basic results like Continuity from Above as well as opening the consideration of signed measures---we have that any given $\T$ embeds into an inverse monoid $\Q$ referred to as \emph{quantity space}.

\begin{thm}[\textbf{The Embedding Theorem}]\label{emtheo}
	Let $(X,S,\A) \in \SM$.  Then there exists an inverse monoid $\Q(X,S,\A)$ such that the type monoid $\T(X,S,\A)$ embeds into $\Q(X,S,\A)$.  Moreover, $\Q(X,S,\A)$ is commutative, partially ordered, and its idempotents form a distributive lattice.
\end{thm}

Quantity spaces $\Q$ are distributive lattices of groups called \emph{isotropy groups} with extra structure allowing one to sum elements from different isotropy groups. This view has a useful interpretation. The isotropy groups represent all the different possible scales of measurement. The lattice ordering specifies the relative sensitivity of each scale of measurement. And adding two quantities of different scales returns the sum of their images in the finest scale less sensitive than both of their respective original scales.

Together, Theorems~\ref{bigtheorem} \&~\ref{emtheo} allow us to establish some basic results of measure theory in much wider generality. In particular, we will demonstrate that Tarski measure has the following properties.
\begin{thm}\label{basics}
	Let $(X,S,\A)$ be a stationarily measurable space and denote its Tarski measure $m: \A \to \T$.
	\begin{description}
		\item[Monotonicity:]
			If $A,B \in \A$ and $A \subseteq B$, then $m(A) \preceq m(B)$.
		\vspace{7pt}
		\item[Subadditivity:]
			If $\{ A_i \}_{i=1}^{\infty} \subseteq \A$, then $m\left( \bigcup\limits_{i=1}^{\infty} A_i \right) \preceq \sum\limits_{i=1}^{\infty} m(A_i)$.
		\item[Continuity from below:]
			If $\{ A_i \}_{i=1}^{\infty} \subseteq \A$ and $A_1 \subseteq A_2 \subseteq \cdots$, then $\varinjlim m(A_i) = m\left( \bigcup\limits_{i=1}^{\infty} A_i \right)$
		\item[Continuity from above:]
			If $\{ A_i \}_{i=1}^{\infty} \subseteq \A$ and $A_1 \supseteq A_2 \supseteq \cdots$, then
			$\varprojlim m(A_i) = m\left( \bigcap\limits_{i=1}^{\infty} A_i \right) + e$
			\newline provided $\exists N \in \N$ and $e \in \E(\T)$ such that for all $n \geq N$, $m(A_n) \in \T_e$ the isotropy monoid of $e$.
	\end{description}
\end{thm}

It is important to note the change in the statement of continuity from above.  In the classical context, the caveat reads something like ``...provided $m(A_N) < \infty$ for some $N$." Here the condition is that the decreasing sequence be eventually in a \emph{particular scale of measurement} with the isotropy monoids corresponding to scales of measurement. The subtlety is washed out in classical theory because there are only two scales of measurement: the finite stuff and infinity.  So the condition of eventually being in a particular scale is truncated to being eventually less than infinity. This condition is naturally suggested by the proof which restates the value $m\left( \bigcap\limits_{i=1}^{\infty} A_i \right)$ as a difference.  Just as in $\overline{\R}$ (what's $\infty - \infty$?) differences are defined relative to a particular scale of the quantity space. So there is an obstruction to proving the theorem: we need to be able to take differences. Thus, we add a condition ensuring that said differences can be taken! In addition, there is an unexpected `$ + \hspace{1ex} e $'.  In the classical context, $e = 0$ so it does not appear.  If in addition, we were to require that the type of $ \bigcap\limits_{i=1}^{\infty} A_i $ also be in the isotropy monoid $\T_e$ then the more familiar result would hold by Lemma~\ref{idembtw}.  However, we see no reason to make the additional assumption here.

The last of our main results relates the collection of classical stationary measures, $\M(X; \overline{\R}^+)$ on a stationarily measurable space $(X,S,\A)$ to a certain collection of $\sigma$\emph{-homomorphisms} which are monoid homomorphisms that commute with countably infinite sums in addition to finite sums e.g. $f$ is a $\sigma$-homomorphism if and only if $f \big( \sum\limits_{j \in J} \alpha_j \big) = \sum\limits_{j \in J} f(\alpha_j)$ for countable indices $J$.

\begin{thm}\label{meascoor}
	$\M(X;\overline{\R}^+)$ is in one-to-one correspondence with the collection of $\sigma$-homomorphisms mapping the completed isotropy monoids $\overline{\T_e}$ into $\overline{\R}^+$.
\end{thm}

The formalism within the previous theorem statement is laid out in Section~\ref{hierarchical} with this theorem coming as a consequence of Proposition~\ref{extmeas}.  However, the content is simple.  Unlike Tarski measure, classical measures can only describe a single scale of measurement. This is a consequence of the fact that $\overline{\R}^+$ has two idempotents $\{ 0, \infty \}$ and $\infty$'s associated isotropy monoid is trivial. Because of this, when there is a stationary classical measure $\mu$ on $(X,S,\A) \in \SM$, $\mu$ must specify a particular scale within the type monoid. This observation allows us to parametrize the space of stationary measures $\M(X; \overline{\R}^+)$ in terms of $\sigma$-homomorphisms out of the completed scales of measurement $\{ \overline{\T_e} : e \in \E(\T) \}$.  In particular, this serves as a formal algebraic connection between Tarski measure and classical measure theory.

\subsection{Discussion of Theorem~\ref{bigtheorem}} 
The proof of this Theorem relying chiefly on a somewhat lengthy construction, will be given in Section~\ref{typeconst}.  Here we comment on its contents and touch on the expectation that the type functor should be co-functorial.
 
The object map taking $(X,S,\A)$ to the collection of \emph{types} $\T(X,S,\A)$ is realized by first expanding the $\sigma$-algebra $\A$ to account for multiple and/or overlapping sets and then quotienting this expanded collection $\overline{\A}$ by an equivalence relation. In the parlance of Section~\ref{typedef}, $\T(X,S,\A)$ is given by the quotient:
\begin{equation}
	\T(X,S,\A) := \overline{\A} \big/  \eqsim_S
\end{equation}
The equivalence relation encodes exactly and only that collection of equations which a stationary measure must satisfy. Consequently, the type of a measurable set $A \in \A$ is the collection of all sets reachable by partitioning $A$ into at most countably many measurable pieces and acting on those pieces independently by $S$.

The co-functoriality of the morphism map also warrants comment. Recall that classical measures pushforward.  That is, given measured spaces $(X,\A,\nu)$,  $(Y,\B,\mu)$ and an $(\A,\B)$-measurable map $f:X \to Y$ the pushforward of $\nu$, $f_{\ast}\nu: \B \to \nu(\A) \subseteq [0,\infty]$ is defined by $(f_{\ast}\nu) := \nu \circ f^{-1} (B)$.
Thus, the pushforward measure assigns to each measurable subset of $Y$ a real number in the image of $\nu$. A natural question arises: Is there a map $f^{\ast}$ such that $\nu \circ f^{-1} = f^{\ast} \circ \mu$?\footnote{The existence of such a map is precisely the content of the commutivity property in part~\ref{propmorph} of Theorem~\ref{bigtheorem}.}  In other words, can we make the following diagram commute?
\begin{equation}\begin{tikzcd}
	  (X,\A,\nu)	\ar{d}{f}
	& \A			\ar{r}{\nu}
	& \nu(\A) 		\subseteq [0,\infty]
	\\
	  (Y,\B,\mu)
	& \B 			\ar{r}{\mu} \ar{u}{f^{-1}}
					\ar{ur}{f_{\ast} \nu}
	& \mu(\B)		\ar[dashed]{u}[swap]{f^{\ast}}
	\\
\end{tikzcd}\end{equation}

Intuitively, this is a question about whether or not the measurable function $f$ not only relates $\B$ to $\A$ but also induces a comparison between the measurements $\mu(B)$ and $\nu(A)$ compatible with the measure $\mu$.  Does the relation between sets $f^{-1}$ induce a relation between measures of sets in a `nice' fashion? Putting aside the meaning of such an $f^*$ we see quite trivially because measurable maps are defined in terms of their preimage that $f^{\ast}: \nu(\A) \leftarrow \mu(\B)$ should map \emph{in the opposite direction} as $f$. So for such a map we would expect the order of composition to reverse just as it does for preimages of compositions i.e. $(f \circ g) \longmapsto (f \circ g)^{\ast} = (g^{\ast} \circ f^{\ast})$. That is, should such a construction exist, without thinking, we expect it to be co-functorial.

\begin{ex}
	The Radon-Nikodym Theorem for classical measures provides us with a concrete example in which this diagram can be completed. Consider a measurable space $(X,\A)$ with measurable map $f: X \to X$ and a pair of $\sigma$-finite measures $\mu$ and $\nu$ on $(X,\A)$ where $\nu \ll \mu$ and $\nu(f^{-1}N) = 0$ whenever $N$ is $\mu$-null.  Then, using the Radon-Nikodym derivative $\frac{\operatorname{d}\!\nu}{\operatorname{d}\!\mu}$, we may define $f^{\ast}$ to be the monoid homomorphism given by sending $\mu(A) = \int\limits_A \, {\operatorname{d}\!\mu} \longmapsto \int\limits_{f^{-1}(A)} \frac{\operatorname{d}\!\nu}{\operatorname{d}\!\mu} \, \operatorname{d}\!\mu$ and observe that;
	\begin{equation}
		f_{\ast}\nu(A)
		=
		(\nu \circ f^{-1})(A)
		=
		\int\limits_{f^{-1}(A)} \frac{\operatorname{d}\!\nu}{\operatorname{d}\!\mu} \, \operatorname{d}\!\mu
		=
		(f^{\ast} \circ \mu)(A).
	\end{equation}
\end{ex}

\section{Construction of the Type co-Functor}\label{typeconst}
Within this Section we prove Theorem~\ref{bigtheorem} and in so doing, construct Tarski measure for an arbitrary stationarily measurable space. Section~\ref{background} consists of some motivating background using classical stationary measures. For reference, the reader can consider Lebesgue measure on $R^n$ which is stationary with respect to the Euclidean group or alternatively, counting measure on an countable set $I$ which is stationary under permutations $\rho \in Aut(I)$. The concept of equidecomposability is introduced in Section~\ref{equidecomp}. This is an equivalence relation on the collection of measurable sets that encodes necessary equalities of measure based on stationarity and countable additivity.  In Section~\ref{typedef}, equidecomposability is used to define the Tarski type monoid of a stationarily measurable space. Section~\ref{mormap} relates morphisms of stationarily measurable spaces to monoid morphisms between their associated type monoids. Finally, in Section~\ref{tarskimeas} we define Tarski measure and conclude the proof of Theorem~\ref{bigtheorem}.

\subsection{Background}\label{background}
Tarski introduced equidecomposability to encode the necessary equivalences in measure between sets together with the algebraic laws a measure must satisfy. It is foundational in the study of paradoxical decompositions such as the famous Banach-Tarski Theorem, investigations into the existence or non-existence of measures with certain properties, and in the theory of Amenable Groups.

The construction itself follows from a careful interpretation of the properties of a measure. Consider an $S$-stationary classical measure space $(X,S,\A,\mu)$ where $S$ is an inverse monoid.  The first observation one should make is that it is not uncommon to make implicit use of objects that are not in our $\sigma$-algebra.  In the identity: $\mu(P) + \mu(Q) = \mu(P \triangle Q) + 2\mu(P \cap Q)$ we exploit the ring structure of $\overline{\R}^+$ to encode, using a `2', a more elementary geometric fact; the left hand side of the equation describes two copies of the set $(P \cap Q)$.  Using coproducts we can encode the identity without appealing to multiplication.\footnote{The coproduct in $\Set$ is the disjoint union.}

\begin{equation}\label{symdiff}
	\mu(P) + \mu(Q) = \mu(P \triangle Q) + \mu\left( (P \cap Q) \coprod (P \cap Q) \right) = \mu\left( P \coprod Q \right)
\end{equation}
\\
Further, this approach matches the intuitionist view of having two `overlapping copies' of $(P \cap Q)$.  The problem, of course, is that $\A$ typically lacks coproducts.  For example, if $P$ is a non-empty measurable set in $\A$ then $P \coprod P \not\in \A$.  However, the lack of coproducts in $\A$ isn't much of an issue as we may view $\A$ as a subcategory of $\Set$ (which has coproducts) and so generate $\overline{\A}$ the minimal subcategory of $\Set$ containing the objects of $\A$ and having colimits of type $\catJ$ where $\catJ$ is a countable discrete category.\footnote{Tarski referred to $\overline{\A}$ as the \emph{disjunctive Algebra} associated with $\A$.} Thus, countable additivity may be expressed more simply as: ``$\mu\big( \coprod\limits_j P_j \big) = \sum\limits_j \mu(P_j)$ for $\{P_j\}_{j \in J} \subseteq \overline{\A} \,$".

\begin{lem}\label{rep}
	Any object $P \in \overline{\A}$ may be represented as $P = \coprod\limits_{j \in J} P_j$ with $P_j \in \A$ for all $j$ \emph{for some} countable index set $J$. Moreover, this representation is unique up to coproducts appearing within $\A$ e.g. measurable partitions.
\end{lem}

This lemma justifies our laxity in treating the measure $\mu$'s domain ambiguously as either $\A$ or $\overline{\A}$ since, together with countable additivity, $\mu: \A \to \overline{\R}^+$ uniquely extends to an $S$-stationary measure $\mu: \overline{\A} \to \overline{\R}^+$ where $S$ acts on $\overline{\A}$ by mapping $s^{-1}(\coprod\limits_{j \in J} P_j) \mapsto \coprod\limits_{j \in J} s^{-1}P_j$.

So we have at our disposal, two measure preserving operations: (i) we may partition a measurable set into countably many measurable pieces via coproducts and (ii) we may act on measurable sets by transformations from $S$.  Used in conjunction, these two operations constitute a coarser invariant of the measure $\mu$ than $S$-stationarity alone because it allows one to act on the members of a coproduct independently of one another. For example, let $P \coprod Q \in \overline{\A}$ and $s,t \in S$ and consider the following figure.

\begin{figure}[h]
\begin{tikzpicture}
\draw[black,thick,fill=gray!20] (0,0)  circle (.5);
\draw[black,thick,fill=gray!20] (10,0) circle (.5);
\draw (0,1)   node {$P$};
\draw (10,-1) node {$Q$};

\draw[black,thick,fill=gray!80] (5,0) circle (.5) node[auto]{$\coprod$};
\draw (5,1) node {$s^{-1}P$};
\draw (5,-1) node {$t^{-1}Q$};

\draw (1,1)  edge[->,bend left] node[auto]{$s^{-1}$} (4,1) ;
\draw (9,-1) edge[->,bend left] node[auto]{$t^{-1}$} (6,-1);

\end{tikzpicture}
\caption{A two piece equidecomposition.}
\label{fig1}
\end{figure}
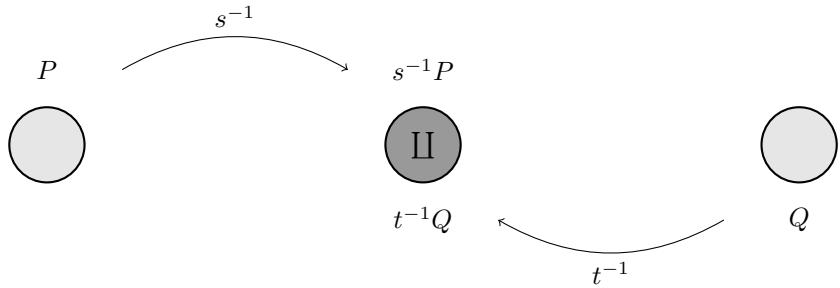
\noindent By equation~\eqref{symdiff} and the $S$-stationarity of $\mu$, the measure of the sets in Figure~\ref{fig1} is preserved.  Specifically, we have the equation:
\begin{equation}\label{equiex}
	\mu\left(P \coprod Q\right) = \mu(P) + \mu(Q) = \mu(s^{-1}P) + \mu(t^{-1}Q) = \mu\left(s^{-1}P \coprod t^{-1}Q \right).
\end{equation}
In this discussion of the measure $\mu$ we have scrupulously avoided using the fact that the target was $\overline{\R}^+$. This is fortunate, as modern measure theory includes projection-valued measures with values given by self-adjoint projections on a Hilbert space as well as positive operator valued measures neither of which are describable using real numbers alone.

\subsection{Equidecomposability}\label{equidecomp}
We observed in the previous section that an $S$-stationary classical measure $\mu$ is in fact invariant with respect to a coarser relation referred to as equidecomposability. The specific example given by equation~\eqref{equiex} demonstrates how equidecomposability differs from and extends $S$-similarity.  The general definition follows.
	
\begin{defn}
	Let $S$ be an inverse monoid, $\Omega$ be an $S$-set, and $\M \subseteq 2^\Omega$ be a collection of subsets closed under the induced action of $S$ (by pullbacks $s^{-1}: \M \to \M$). We say that $P,Q \in \M$ are $(S,\M)$\emph{-equidecomposable} and write $P \sim_S Q \iff$ there exist countable collections $(P_j), (Q_j) \subseteq \M$ with $P = \coprod\limits_j P_j$ and $Q = \coprod\limits_j Q_j$ and $(s_j), (t_j) \subseteq S$ such that $s^{-1}_j P_j = t^{-1}_j Q_j$.
\end{defn}
	
\begin{rem}\label{wlogidem}
	In the above we may, without loss of generality, assume for each $j$ that one of $s_j$ or $t_j$ is idempotent.
\end{rem}
\begin{proof}
	Observe that if $s^{-1}_j P_j = t^{-1}_j Q_j$ then $(s_j s^*_j)^{-1} P_j = (t_j s^*_j)^{-1} Q_j$.  Taking $s^*_j$ to be the weak inverse of $s_j$ in $S$, the result follows.\footnote{unique elements of an inverse semigroup $S$ satisfying the property that $ss^*s = s$ and $s^*ss^* = s^* $. See Appendix~\ref{sym}}
\end{proof}

\begin{prop}
	$(S,\M)$-equidecomposability is an equivalence relation on $\M$
\end{prop}
	
\begin{proof}
	Reflexivity and symmetry of $S$-equidecomposability are immediate.  Transitivity however, is somewhat more involved. Suppose $P \sim_S Q \sim_S R$. Then there exist $P_i, Q_i, Q'_j, R_j \in \M$ along with $s_i, t_i, u_j, v_j \in S$ for $i,j \in J$ a countable index set, such that:
	$$ P = \coprod\limits_i P_i \hspace{40pt} \coprod\limits_i Q_i = Q = \coprod\limits_j Q'_j \hspace{40pt} R = \coprod\limits_j R_j $$
	and
	$$s^{-1}_i P_i = t^{-1}_i Q_i \hspace{40pt} u^{-1}_j Q'_j = v^{-1}_j R_j$$ 
	By Remark~\ref{wlogidem} we can assume that $t_i, u_j$ are idempotent for all $i$ and $j$. Consequently, as the idempotents of an inverse monoid commute:
	$$( t_i u_j )^{-1}(Q_i \cap Q'_j) = ( u_j t_i )^{-1}(Q_i \cap Q'_j)$$
	So we can partition $Q$ as $\coprod\limits_{i,j} (Q_i \cap Q'_j)$ and use that to relate partitions of $P$ and $R$.  Explicitly,
	$$\begin{array}{lclcl}
		  (s_i u_j)^{-1}P_{ij}
		& :=
		& (s_i u_j)^{-1} \Big( P_i \cap (s_i u_j) (t_i u_j)^{-1}(Q'_j) \Big)
		& 
		& 
		\\
		  
		& =
		& (s_i u_j)^{-1}(P_i) \cap (t_i u_j)^{-1}(Q'_j)
		&
		&
		\\
		  
		& =
		& (t_i u_j)^{-1}(Q_i \cap Q'_j)
		& 
		& 
		\\
		  
		& =
		& (u_j t_i)^{-1}(Q_i \cap Q'_j)
		& 
		& 
		\\
		  
		& =
		& (v_j t_i)^{-1}(R_j) \cap (u_j t_i)^{-1}(Q_i)
		\\
		  
		& =
		& (v_j t_i)^{-1} \Big( R_j \cap (v_j t_i) (u_j t_i)^{-1}(Q_i) \Big)
		& =:
		& (v_j t_i)^{-1}R_{ij}
		\\
	\end{array}$$
	Since $P = \coprod\limits_{i,j} P_{ij}$ and $R = \coprod\limits_{i,j} R_{ij}$ this is a realization of $P \sim_S R$.
\end{proof}

The following proposition summarizes the relationship between $(S,\A)$-equidecomposability and $(S,\overline{\A})$-equidecomposability.  Because in both instances the action is by $S$, we introduce the notation ``$\, \eqsim_S$" to distinguish $(S,\overline{\A})$-equidecomposability from $(S,\A)$-equidecomposability.

\begin{prop}
	Given $P, Q \in \A$ we have $P \sim_S Q \implies P \eqsim_S Q$
\end{prop}
\begin{proof}
	The only countable coproducts in $\A$ are partitions into sets in $\A$ and all such coproducts are by definition in $\overline{\A}$. Hence, this follows directly from the definition of equidecomposability and of $\overline{\A}$.
\end{proof}
	
It is trivial to check that for a given $S$-stationary measure $\mu: \A \to \overline{\R}^+$; we have that $P \eqsim_S Q \implies \mu(P) = \mu(Q)$. Substantially, this is the content of equation~\eqref{equiex}. The converse however is not true of $(S,\overline{\A})$-equidecomposability for at least two reasons: (i) There may be some degeneracy in the $S$-action. and (ii) The $\sigma$-ideal of null sets prevents it. To illustrate, consider the following examples.
\begin{ex}\label{evenodd}
	Take $(X,S,\A) = (\Z,\Z^2,2^\Z)$ where $\Z^2$ acts on $\Z$ by:
	\begin{equation}
		(p,q)(n):=
		\left\{
		\begin{array}{ll}
			  n + 2p
			& : n \mbox{ even}
			\\
			  n + 2q
			& : n \mbox{ odd}
		\end{array}
		\right.
	\end{equation}
	Then the countably-additive $S$-stationary measure defined by taking $\mu(\{ 0 \}) := 1$ and $\mu(\{ 1 \}) := \frac{1}{2}$ has the property that $\mu(\{ 0 \}) = \mu(\{ -1,1 \})$ but clearly $\{ 0 \} \not\eqsim_S \{ -1,1 \}$ since the action preserves parity.
\end{ex}
	
\begin{ex}
	If for a given stationarily measured space $(X,S,\A,\mu)$ there exists a $\mu$-null set $N \neq \varnothing$ and $S$ does not contain an absorbing element, then $\varnothing \not\eqsim_S N$.
\end{ex}
	
By considering equidecomposability up to the difference of a null set, a partial resolution is available in the case of Haar measure on a locally compact topological group equipped with its Borel $\sigma$-algebra.  The following theorem given in \cite{wagon1993} summarizes:
	
\begin{thm}
	For the stationarily measured space $(X,S,\A,\mu):=(G,G,\B_G,\mu)$ where $\mu$ is (left) Haar measure. $\mu (P) = \mu(Q) \iff \exists N,N' \in \mathcal{N}$ such that $(P \setminus N) \eqsim_G (Q \setminus N')$ where $\mathcal{N}$ is the $\sigma$-ideal of $\mu$-null sets.
\end{thm}
	
As in equation \eqref{symdiff}, the extended framework using $\overline{\A}$ allows one to make the identification of ``two copies of $P$" with $P \coprod P$.  Together with the notion of $(S,\overline{\A})$-equidecomposability, this is a formalization of the historical studies of paradoxical decompositions and so-called ``paradoxes of measure" \cite{wagon1993}.  A simple example of this idea dating back to Galileo namely, that: \emph{The set of whole numbers is equinumerous with two copies of itself}, translates in this language to:
\begin{rem}
	For $(X,S,\A):=(\Z,\Z,2^\Z)$, we have the relationship $\Z \eqsim_S (\Z \coprod \Z)$.
\end{rem}
Likewise, viewed as a measurable object in $(\R,E,\B_\R)$ where $E$ is the Euclidean group, we have again that $\Z \eqsim_E (\Z \coprod \Z)$. An interesting feature of classical measure theory is that in the first case, this equivalence serves to prove that an invariant measure must give the object $\Z$ infinite measure whereas in the latter case, it is used (at least with respect to $\sigma$-finite measures) to prove that the object $\Z$ must be null.  At the core of both of these assertions is the idea that when the target of our measures is $\overline{\R}^+$, any set which is equidecomposable with two copies of itself must have measure $0$ or $\infty$ because these are the only idempotents in $\overline{\R}^+$.  That is to say, the conclusion is not intrinsic; it is a byproduct of our assumptions on the model for a measurement.
	
\subsection{The Tarski Type Monoid}\label{typedef}
From Lemma~\ref{rep} and the preceding section we have the observation that $S$-stationary measures of $\A$ are in one to one correspondence with $S$-stationary measures of $\overline{\A}$ and the values of such measures are invariant on the $(S,\overline{\A})$-equidecomposability classes of $\overline{\A}$. Here we establish that the quotient $\overline{\A} \big/ \eqsim_S$ inherently carries the structure of a partially ordered commutative monoid.
	
\begin{defn}
	The \emph{type monoid} of a stationarily measurable space $(X,S,\A)$ denoted $\T(X,S,\A)$ or in clear contexts, simply $\T$, is given by the quotient:
	\begin{equation}
		\T(X,S,\A) := \overline{\A} \big/ \eqsim_S = \{ [P] : P \in \overline{\A} \}
	\end{equation}
	together with the binary operation ``+" defined:
	\begin{equation}
		[P] + [Q] := \big[ P \textstyle\coprod Q \big]
	\end{equation}
\end{defn}
	
Throughout, we refer to these equivalence classes as \emph{types}, members of a given equivalence class as \emph{representatives}, and equidecompositions demonstrating relationships between types as \emph{realizations}.
\begin{prop}\label{typemon}
	The type monoid $(\T,+)$ of a stationarily measurable space is a well-defined commutative monoid with identity $0:= [\varnothing]$.
\end{prop}
\begin{proof}
	It suffices to demonstrate that if $P \in [R]$, then $[P]+[Q] = [R]+[Q]$. By definition, $[P]+[Q] = [P \coprod Q]$ and $[R]+[Q] = [R \coprod Q]$. Observe:
	\begin{equation}
			\begin{array}{clclclclc}
				[P \coprod Q]
				& =
				& \Big[ \big( \coprod\limits_{j \in J} P_j \big) \coprod Q \Big]
				& =
				& \Big[ \big( \coprod\limits_{j \in J} s^{-1}_j P_j \big) \coprod Q \Big]
				\\
				
				&
				&
				& =
				& \Big[ \big( \coprod\limits_{j \in J} t^{-1}_j R_j \big) \coprod Q \Big]
				& = 
				& \Big[ \big( \coprod\limits_{j \in J} R_j \big) \coprod Q \Big]
				& =
				&[R \coprod Q]
				\\
			\end{array}
	\end{equation}
	for a realization of $P \eqsim_S R$ via $P = \coprod\limits_{j \in J} P_j$, $R = \coprod\limits_{j \in J} R_j$ with $s^{-1}_j P_j = t^{-1}_j R_j$. Finally, observe that $[P] + [\varnothing] = [P]$ for all types $[P]$ since $P \coprod \varnothing = P$ in $\overline{\A}$.
\end{proof}
	
That the canonical preorder on commutative monoids ($\alpha \leq \beta \iff \exists \gamma$ such that $\alpha+\gamma=\beta$) is actually, a partial ordering on $\T$ follows from a slight modification of Banach's Schr\"{o}der-Bernstein Theorem given in \cite{wagon1993} and originally \cite{banach1924}. However, we'll need to establish a few facts before proving the theorem.

First, observe that for all $s \in S$ we have, $\d(x) = s^*(y) \iff s(x) = \r(y)$ where $\d = \d_s := s^*s$ is the \emph{domain idempotent of} $s$ and $\r = \r_s:= s s^*$ is the \emph{codomain idempotent of} $s$ are in $\E(S)$ the meet semilattice of idempotents of $S$.  Consequently, $s$ restricts to a bijection between $\d X$ and $\r X$ with inverse function $s^*$ and we have the following:
	
\begin{equation}\label{weakinv}\begin{array}{ccc}

	\begin{tikzcd}
		  X		\ar{dr}{s}
		  		\ar{d}{\d}
		& 
		\\
		  \d X 	\ar[yshift = +0.5ex]{r}{s}
		& \r X 	\ar[yshift = -0.5ex]{l}{s^*}
		\\
		& X		\ar{ul}{s^*}
				\ar{u}{\r}
		\\
	\end{tikzcd}
	
& \hspace{10ex} &

	\begin{array}{lcl}
	  s^{-1}(A)
	& = 
	& s^{-1}(A \cap \r X) \sqcup s^{-1}(A \setminus \r X)
	\\
	  
	& =
	& s^{-1}(A \cap \r X)
	\\
	  
	& =
	& s^*(A \cap \r X)
	\end{array}
\end{array}\end{equation}
	
\begin{rem}\label{bi}
	If $p = q \in \T$ then for any $P, Q \in \overline{\A}$ realizing the equality, we may restrict to a bijective realization.
\end{rem}
\begin{proof}
	This follows from the special case where $[P_0] = p = q = [Q_0]$ via $s^{-1}P_0 = t^{-1}Q_0$.  Here we may just define $P := (P_0 \cap \r_sX)$ and $Q := (Q_0 \cap \r_tX)$ and observe that $P_0 \eqsim P \eqsim Q \eqsim Q_0$ with the middle relation achieved by the bijection $t^{-1}s^*: P \to Q$. In general, just apply this same construction to each of the indices of a realization of $p = q$.
\end{proof}

\begin{lem}\label{consist}
	For members of $\overline{\A}$ the following two consistency properties with respect to equidecomposability hold:
	\begin{description}
		\item[refining:]   If $P \eqsim Q$ via bijection $\tau: P \to 
		Q$ then $R \sim \tau R$ whenever $R \subseteq P$ in $\overline{\A}$.
		
		\item[additive:] $P_1 \eqsim Q_1$ and $P_2 \eqsim Q_2 \implies P_1 \coprod P_2 \eqsim Q_1 \coprod Q_2$
	\end{description} 
\end{lem}
\begin{proof}
	That $(S,\overline{\A})$-equidecomposability is refining follows from the measurability of the $S$-action on $X$. The additivity follows trivially from the definition of equidecomposability.
\end{proof}

\begin{thm}[\textbf{Schr\"{o}der-Bernstein Property}]\label{bsb}
	The canonical preorder on $\T = \overline{\A} \big/ \eqsim_S$ given by $\alpha \preceq \beta \iff \exists \gamma$ such that $\alpha + \gamma = \beta$ is a partial order.
\end{thm}
\begin{proof}
	We need only prove that``$\preceq$" is antisymmetric.  That is, if for $p,q \in \T$, $p \preceq q$ and $p \succeq q$ then $p = q$.  We will prove this directly.
	
	Suppose $p \preceq q$ and $p \succeq q$.  Then for some representatives we have $P \coprod \square \eqsim Q$ which implies $P \eqsim Q'$ for some $Q' \subseteq Q$.  Likewise, $Q \eqsim P'$ for some $P' \subseteq P$.  Moreover, by Remark~\ref{bi} we may assume these are equidecomposed via bijections $\tau: P \to Q'$ and $\rho: P' \to Q$.
	
	Define a sequence of sets in $\overline{\A}$ by taking $R_0 := P \setminus P'$ and $R_{n+1} := (\rho^{-1} \circ \tau)R_n$. Let $R := \bigcup\limits_{n=0}^\infty R_n$ and observe that $(\rho^{-1} \circ \tau)R = R$ which implies $\rho R = \tau R$ and so $\rho(P \setminus R) = Q \setminus \tau R$. Further as, $P \setminus R \subseteq P'$ and $\rho(P \setminus R) = Q \setminus \tau R$ by the refining property we have that $P \setminus R \eqsim Q \setminus \tau R$. Similarly, since $R \subseteq P$ we have $\tau R \eqsim R$.  Since, '$\eqsim_S$' is additive;
	\begin{equation}
		\textstyle
		P = (P \setminus R) \coprod R \eqsim (Q \setminus \tau R) \coprod \tau R = Q
	\end{equation} 
	Finally, as the constructed $R$ is in $\overline{\A}$, this is a realization of $p = q \in \T$.  Therefore, the canonical preorder on commutative monoids is a partial ordering of $\T$.
\end{proof}

\begin{thm}[\textbf{$\N$-Cancellation}]\label{ncan}
	Given $\alpha, \beta \in \T$ and $n \in \N$; $n \alpha = n \beta \implies \alpha = \beta$
\end{thm}
\begin{proof}
	This is an immediate consequence of Theorem 1.12 of~\cite{chuaqui1969}.
\end{proof}

\begin{cor}\label{idemimport}
	If $\alpha \in \T$ and $n \in \N$ are such that $(n+1) \alpha \preceq n \alpha$, then $2 \alpha = \alpha$.
\end{cor}
\begin{proof}
	Substituting the hypothesis into itself yields the inequality $(n-1)$ times, we have that:
	\begin{equation}\begin{array}{lclclcl}
		  n \alpha
		& \succeq
		& n \alpha + \alpha
		\\
		  
		& \succeq
		& n \alpha + 2 \alpha
		\\
		  
		& \vdots
		& 
		\\
		  
		& \succeq
		& n \alpha + n \alpha
		& =
		& n(2 \alpha)
		& \succeq
		& n \alpha
	\end{array}\end{equation}
	which implies that $n \alpha = n(2 \alpha)$. Therefore by Theorem~\ref{ncan}, $\alpha = 2 \alpha$ as desired.
\end{proof}
Theorem~\ref{ncan} and its Corollary should be interpreted geometrically in terms of their meaning for representations of the involved types.  If for some type $\alpha$ and natural number $n$ we have $n \alpha = (n+1) \alpha$ it is the case that $n$ disjoint copies of a representative $A$ of $\alpha$ collectively can be equidecomposed into $(n+1) $ copies of $A$.  The theorem then implies that $A$ itself can be equidecomposed into $2$ copies of $A$. This violates an, as it turns out incorrect, intuition about size: `partition and symmetry can't make more of something'. Fortunately, Theorem~\ref{ncan} does guarantee that the ways in which we can `make more of something' are all of a kind. Sets admitting some paradoxical result in measure admit a duplication.  The previous two results demonstrate that within this formalization of measure, it is only necessary to consider one kind of paradox: duplication.  This observation plays a key role in Section~\ref{structure} allowing us to decompose Type monoids according to those types which can be duplicated e.g. the idempotents.

\subsection{The Morphism Map}\label{mormap}
Having constructed the Tarski Type Monoid, we now turn our attention to morphisms in $\SM$. Recall, a morphism $f: (X,S,\A) \to (Y,T,\B)$ is an $(\A,\B)$-measurable function with the property that for all $t \in T$ there exists $s \in S$ such that $f^{-1} \circ t^{-1} = s^{-1} \circ f^{-1}$.  Such an $f$ determines a map $\overline{f}^{-1}:\overline{\B} \to \overline{\A}$ defined by mapping $ B = \coprod\limits_j B_j \mapsto \coprod\limits_j f^{-1}B_j \in \A$.  Moreover, the symmetry condition, that orbits pullback to orbits, precisely guarantees that this map induces a map of types.
\begin{defn}
	Given $f \in mor(\SM)$ define $\T f: \T(Y,T,\B) \to \T(X,S,\A)$ to be the map taking $[B] \mapsto [\overline{f}^{-1}B]$.
\end{defn}

\begin{prop}
	For every $f \in mor(\SM)$ the map $\T f: \T(Y,T,\B) \to \T(X,S,\A)$ is well-defined, partial order-preserving monoid $\sigma$-homomorphism.
\end{prop}
\begin{proof}
	First we need show that if $[P] = [Q]$ then $\T f [P] = \T f [Q]$.  The following diagram applied to each part of a realization of $P \eqsim_T Q$ establishes this fact. Let $u_j,v_j \in T$ such that $u_j^{-1} P_j = v_j^{-1} Q_j$. Then for some $s_j := f^*(u_j)$ and $f^*(v_j):= t_j \in S$ we have:
	\begin{equation}\begin{tikzcd}
		  P_j
		  \ar[end anchor = {[xshift = +1.75ex]}]{r}[xshift = +.5ex]{u_j^{-1}}
		  \ar{dd}[swap]{\overline{f}^{-1}}
		& \hspace{1.75ex} u_j^{-1}P_j \hspace{2ex}
		  \ar[yshift = +.25ex, dash]{r}
		  \ar[yshift = -.25ex, dash]{r}
		  \ar{d}{\overline{f}^{-1}}
		&[-4.5ex] \hspace{2ex} v_j^{-1} Q_j \hspace{1.75ex}   
		  \ar{d}{\overline{f}^{-1}}
		& Q_j
		  \ar[end anchor = {[xshift = -1.75ex]}]{l}[swap, xshift = +.5ex]{v_j^{-1}}
		  \ar{dd}{\overline{f}^{-1}}
		\\
		& \overline{f}^{-1} u_j^{-1} P_j
		  \ar[yshift = +.25ex, dash]{r}
		  \ar[yshift = -.25ex, dash]{r}
		&[-4.5ex] \overline{f}^{-1} v_j^{-1} Q_j
		\\
		  \overline{f}^{-1} P_j
		& s_j^{-1} \overline{f}^{-1} P_j
		  \ar[dash, start anchor = {[yshift = +1.75ex]}, end anchor = {[yshift = -1.5ex]}, xshift = +.25ex]{u}
		  \ar[dash, start anchor = {[yshift = +1.75ex]}, end anchor = {[yshift = -1.5ex]}, xshift = -.25ex]{u}
		  \ar[leftarrow]{l}[swap, xshift = +.5ex]{s_j^{-1}}
		& t_j^{-1} \overline{f}^{-1} Q_j
		  \ar[dash, start anchor = {[yshift = +1.75ex]}, end anchor = {[yshift = -1.5ex]}, xshift = +.25ex]{u}
		  \ar[dash, start anchor = {[yshift = +1.75ex]}, end anchor = {[yshift = -1.5ex]}, xshift = -.25ex]{u}
		  \ar[leftarrow]{r}[xshift = +.5ex]{t_j^{-1}}
		& \overline{f}^{-1} Q_j
	\end{tikzcd}\end{equation}
	So, $\overline{f}^{-1} P \eqsim_S \overline{f}^{-1} Q$ which implies $\T f [P] = \T f [Q]$.  Hence, $\T f$ is well-defined. 
	
	Next, we show that $\T f$ is indeed a monoid $\sigma$-homomorphism.  Note that, as the partial orders on $\T(X,S,\A)$ and $\T(Y,T,\B)$ are the canonical preorders, this implies that $\T f$ is order-preserving.  Observe:
	\begin{equation}\begin{array}{lclclclcl}
		  \T f \Big( \displaystyle\sum\limits_j \big[ P_j \big] \Big)
		& =
		& \T f \Big( \displaystyle\coprod\limits_j P_j \Big)
		& =
		& \Big[ \overline{f}^{-1} \Big( \displaystyle\coprod\limits_j P_j \Big) \Big]
		\\
		  
		&
		& 
		& =
		& \Big[ \displaystyle\coprod\limits_j \overline{f}^{-1}P_j \Big]
		& =
		& \displaystyle\sum\limits_j \big[ \overline{f}^{-1} P_j \big]
		& =
		& \displaystyle\sum\limits_j \T f \big[ P_j \big]
	\end{array}\end{equation}
\end{proof}

\begin{thm}
	The map $\T: \SM \to \Mon$ given by mapping $(X,S,\A) \mapsto \T(X,S,\A)$ and $f \mapsto \T f$ is a cofunctor.
\end{thm}

\begin{proof}
	First, consider the identity map $id_X : (X,S,\A) \to (X,S,\A)$ and observe that $\T(id_X) = id_{\T(X,S,\A)}$ since $\T(id_X)$ sends type $[A] \mapsto [A]$ by definition. Second, suppose that we have morphisms $(X,S,\A) \overset{f}{\to} (Y,T,\B) \overset{g}{\to} (Z,R,\C)$.  Then we must show that $\T(g \circ f) = \T f \circ \T g$ but this is a direct consequence of the fact that $\overline{(g \circ f)}^{-1} = \overline{f}^{-1} \circ \overline{g}^{-1}$.
\end{proof}

\subsection{Tarski Measure}\label{tarskimeas}
As may be expected by now, Tarski measure is given by passing a measurable set to its equidecomposability class.  Specifically,
\begin{defn}
	Given $(X,S,\A) \in \SM$ we define $m_{(X,S,\A)}: \A \to \T(X,S,\A)$ as the map taking a measurable set $A$ to $[A]$, the $(S,\overline{\A})$-equidecomposability class of $A$.
\end{defn}

\begin{prop}
	Let $m = m_{(X,S,\A)}$ be Tarski measure on $(X,S,\A)$. Then $m$ has the properties stated in Theorem~\ref{bigtheorem}.  Namely;
	\begin{description}
		\item[Null Identity:]
			$m(\varnothing) = 0$ the additive identity in $\T(X,S,\A)$ 
		\item[Countable Additivity:]
			If $\{A_{i}\}^{\infty}_{i=1}$ is a sequence of disjoint sets in $\A$ then $m\left( \bigcup\limits_{i=1}^{\infty} A_{i} \right) = \sum\limits_{i=1}^{\infty} m(A_{i})$
		\item[Stationarity:]
			$m(A) = m(s^{-1}A)$ for all $s \in S$
	\end{description}
\end{prop}

\begin{proof}
	These follow easily from the established properties of equidecomposability.  By Propostion~\ref{typemon}, $m(\varnothing) = [\varnothing]$ is the additive identity in $\T$.  Countable Additivity follows by definition from the addition law in $\T$ since for pairwise disjoint measurable sets $A_i$ we have $m\big( \bigcup\limits_{i=1}^{\infty} A_{i} \big) = \big[ \coprod\limits_i A_i \big] = \sum\limits_i [A_i] = \sum\limits_i m(A_i)$.  Finally, stationarity follows since $A \eqsim_S s^{-1}A$.
\end{proof}

Now, having both the type cofunctor and Tarski measure in hand, we may prove the last remaining claim in Theorem~\ref{bigtheorem}, the commutativity of measurement.
\begin{prop}
	Let $m_X: \A \to \T(X,S,\A)$ and $m_Y: \B \to \T(Y,T,\B)$ denote the Tarski measures of $(X,S,\A)$ and $(Y,T,\B)$ respectively. If $f: (X,S,\A) \to (Y,T,\B)$, then $m_X \circ f^{-1} = \T f \circ m_Y$.
\end{prop}
\begin{proof}
	Let $B \in \B$ and consider $m_X \circ f^{-1} B = [f^{-1}B]_X \in \T(X,S,\A)$.  Because $B \in \B$, we have that $f^{-1}B = \overline{f}^{-1}B$.  Hence, $\T f \circ m_Y B = \T f [B] = f^{-1}B$ as claimed.
\end{proof}

Jointly, the results in Sections~\ref{typedef}, \ref{mormap} and \ref{tarskimeas} complete the proof of Theorem 1.

\section{Structure of the Type Monoid}\label{structure}

\subsection{Motivation}
Commonly in Semigroup Theory, it is useful to discuss the structure of the semilattice of \emph{idempotents}. In light of Corollary~\ref{idemimport}, studying the idempotents of type monoids is doubly important. As it stands, a given type monoid $\T := \T(X,S,\A)$ is a very abstracted entity.  Section~\ref{idemlattice} sheds some light on their character by first analysing the collection of idempotent elements of a type monoid $\E(\T):= \{ e \in \T: e+e = e \}$ and then refining our understanding of the addition operation as it relates to the idempotents.

\subsection{The Idempotents form a Distributive Lattice}\label{idemlattice}
Viewed as a category, a lattice is a thin category possessing all finite products and coproducts.  Algebraically, a lattice may be characterized as a being a set $\E$ together with a pair of idempotent, commutative, and associative binary operations $\vee$ and $\wedge$ possessing identities $\top$ and $\bot$ respectively and satisfying the \emph{absorption laws} 
\begin{equation}\begin{array}{ccc}
	  e \vee   (e \wedge f) = e
	& \hspace{50pt}\mbox{ and }\hspace{50pt}
	& e \wedge (e \vee   f) = e 
\end{array}\end{equation}
If in addition, the \emph{distributive laws}
\begin{equation}\begin{array}{ccc}
	  e \vee   (f \wedge g) = (e \vee   f) \wedge (e \vee   g)
	& \hspace{15pt}\mbox{ and }\hspace{15pt}
	& e \wedge (f \vee   g) = (e \wedge f) \vee   (e \wedge g)
\end{array}\end{equation}
hold, then a lattice is called a \emph{distributive lattice}.
\begin{defn}
	Given idempotents $e,f \in \E(\T)$ the \emph{join} of $e$ and $f$, written $e \wedge f$ is defined: $$e \wedge f := \max\Big\{ [E \cap F] : [E] = e \mbox{ and } [F] = f \Big\}$$
\end{defn}

It is simple to verify that $(e \wedge f) \in \E(\T)$. The intuition here is that any overlap $E \cap F$ of realizations $E$ and $F$ can be doubled because the associated types of are idempotent e.g. $[E] = e = [E \coprod E]$ and so the maximal elements should double to themselves and hence be idempotent.

More formally, observe that the maximum is well defined by supposing to the contrary that $m_1$ and $m_2$ are distinct maxima realized by the sets $(E_1 \cap F_1), (E_2 \cap F_2) \in \overline{\A}$.  Then by idempotence of $e$ and $f$ we have that $e = [E_1 \coprod E_2]$ and $f = [F_1 \coprod F_2]$ which implies that $[(E_1 \coprod E_2) \cap (F_1 \coprod F_2)] = [E_1 \cap F_1 \coprod E_2 \cap F_2] = m_1 + m_2$ is also a realization of an intersection of $e$ and $f$.  Consequently, since $m_1$ and $m_2$ are maxima, we have $m_1 \succeq m_1 + m_2 \preceq m_2$.  But by definition $m_1 \preceq m_1 + m_2 \succeq m_2$ and hence $m_1 = m_1+m_2 = m_2$.  Applying this argument to the maximum (e.g. unique maximal element) we see immediately that $(e \wedge f) = 2(e \wedge f)$.  That is, the join is idempotent.

\begin{prop}\label{idemmeet}
	Given an idempotent $m \in \E(\T)$ we have that $m \preceq (e \wedge f) \iff m \preceq e \mbox{ and } m \preceq f$
\end{prop}
\begin{proof}
	The forward implication follows by definition as $(e \wedge f) = [E \cap F] \preceq [E],[F] = e, f$ respectively.  For the reverse suppose  $m \preceq e,f$ but $m \not\preceq (e \wedge f)$.  But this is nonsense as we would have for some representatives that: $e = [E] := [M \coprod K]$ and $f = [F] := [M \coprod K']$ which implies that $m = [M] = [E \cap F]$ which in turn implies that $m \preceq (e \wedge f)$ as desired.
\end{proof}

Now we wish to prove the analogous result for the \emph{join} operation in $\E(\T)$ which is given by addition e.g. $e \vee f := e + f$. To do so we first prove the following useful proposition about the arithmetic of the type monoid.

\begin{lem}\label{idembtw}
	Given $a,b \in \T$ we have $a \preceq f \preceq b$ for some $f \in \E(\T) \iff a+b = b$.
\end{lem}
\begin{proof}
	For the reverse implication, consider the collection $U:= \{ \alpha \in \T: \alpha + b = b \}$. Let $f$ be the unique maximal element of $U$. Since $f$ is maximal and $f+f \in U$ we have that $f \in \E(\T)$ and by construction $a \preceq f \preceq b$. Now we establish the forward direction.  By definition, $a+b \succeq b$ and by hypothesis $b = f+b'$ for some $b'$. So $a+b = a+f+b' \preceq 2f+b' = f+b' = b$.  Therefore, as '$\preceq$' is a partial order, $a+b=b$.
\end{proof}

Indeed, because $\T$ is a cardinal algebra as defined by Tarski~\cite{tarski1949} we can improve this result to the following:
\begin{thm}\label{idemsum}
	Given $a,b \in \T$ we have $a \preceq \sum\limits_{j \in J} a \preceq b \iff a+b = b$ where $J$ is a countable index set. If in addition, $b$ is minimal e.g. for any $c \in \T$ such that $a+c = c$ we have $b \preceq c$ then $b = \sum\limits_{j \in J} a$.
\end{thm}

Together, Lemma~\ref{idembtw} and Theorem~\ref{idemsum} imply by induction that all idempotents are the countable-fold sum of smaller elements of $\T$. Specifically, we have the following Corollary:
\begin{cor}\label{idemsum2}
	The set of idempotents $\E(\T) = \{ \sum\limits_{j \in J} a : a \in \T \}$ and for any idempotent $e$ we have $e = \sum\limits_{j \in J} e$.
\end{cor}

\begin{cor}\label{shortcuts}
	Let $e,f \in \E(\T)$ and $a,b \in \T$.
	\begin{enumerate}[(i)]
		\item $a \preceq f \iff a+f = f$
		
		\item $f \preceq b \iff f+b = b$
		
		\item $e \preceq f \iff e+f = f$
	\end{enumerate}
\end{cor}

\begin{prop}\label{idemjoin}
	Given an idempotents $m,e,f \in \E(\T)$ we have that $(e \vee f) \preceq m \iff e \preceq m \mbox{ and } f \preceq m$.
\end{prop}
\begin{proof}
	Again, the forward implication follows trivially from definitions.  Now assume $e \preceq m$ and $f \preceq m$ for some idempotent $m$. Then by Proposition~\ref{idembtw} we have $e+m = m = f+m$ and in particular, $e \vee f \preceq m$ since $e \vee f + m = e+f+m = m$.
\end{proof}

It is clear that the idempotent $0 = [\varnothing]$ serves as bottom element in $\E(\T)$ e.g. $\bot := 0$.  The top element is given by the countable-fold sum of $[X]$'s.  That is, $\top := \coprod\limits_{j \in J} [X]$ where $J$ is a countable index set.  That $\top$ is in fact an idempotent of $\T(X,S,\A)$ is immediate from Corollary~\ref{idemsum2}.

\begin{prop}
	The lattice of idempotents $\E(\T)$ is distributive. Moreover, $\E(\T)$ is a countably complete lattice and distributes over countable meets and joins as a consequence of $\overline{\A}$'s completeness properties.
\end{prop}
\begin{proof}
By lattice duality we need only prove one of the distributive laws holds.  We prove the second. By Proposition~\ref{idemjoin}, $(e \wedge f) \vee (e \wedge g) \preceq (e \wedge (f+g)) \iff (e \wedge f)+(e \wedge g) \preceq e$ and $(e \wedge f)+(e \wedge g) \preceq (f+g)$.
By definition of the meet or Proposition~\ref{idemmeet}, we have that the meet of two idempotents is less than or equal to either of the individual idempotents. Consequently, the first condition is met because:
$$(e \wedge f)+(e \wedge g) \preceq e + e = e$$
and the second condition holds since:
$$(e \wedge f)+(e \wedge g) \preceq (f + g) $$
Therefore $e \wedge (f \vee g) \succeq (e \wedge f) \vee (e \wedge g)$.

Now we must show that $e \wedge (f \vee g) \preceq (e \wedge f) \vee (e \wedge g)$. Let $E,F,G \in \overline{\A}$ realize the meet $(e \wedge (f \vee g))$.  Then we have:
$$\begin{array}{lclcl}
	  (e \wedge (f \vee g))
	& =
	& \Big[E \cap (F \coprod G)\Big]
	& 
	& 
	\\[8pt]
	  
	& =
	& \Big[E \cap (F \coprod G)\Big] + \Big[E \cap (F \coprod G)\Big]
	& 
	& 
	\\[8pt]
	  
	& =
	& \Big[(E \cap (F \coprod G)) \coprod (E \cap (F \coprod G))\Big]
	& 
	& 
	\\[8pt]
	  
	& =
	& \Big[(E \coprod E) \cap ((F \coprod G) \coprod (F \coprod G))\Big]
	& 
	& 
	\\[8pt]
	  
	& =
	& \Big[(E \coprod E) \cap (F \coprod G)\Big]
	& 
	& 
	\\[8pt]
	  
	& =
	& \Big[(E \cap F) \coprod (E \cap G)\Big]
	& 
	& 
	\\[8pt]
	  
	& =
	& \Big[E \cap F\Big] + \Big[E \cap G\Big]
	& \preceq
	& (e \wedge f) + (e \wedge g)
	\\
\end{array}$$
by definition since the meet is maximal over such intersections.
\end{proof}

Section~\ref{embedtypemon}'s formalism, growing out of this analysis, yields a simple description of the general form of a type monoid: A type monoid $\T$ is a distributive lattice of `scales' $\T_e$ parametrized by the idempotents $e \in \E(\T)$. Each scale is (partially) infinite with respect to scales below and (partially) infinitesimal with respect to scales above. From this we are able to prove the Embedding Theorem by formally adjoining inverses within each scale. The following simple example provides a useful illustration of the scale structures present in type monoids.

\begin{ex}
	Recall the stationarily measurable space $(X,S,\A) = (\Z,\Z^2,2^\Z)$ from Example~\ref{evenodd} where $S$ acts on $X$ by:
	$$
	(p,q)(n):=
		\left\{
		\begin{array}{ll}
			  n + 2p
			& : n \mbox{ even}
			\\
			  n + 2q
			& : n \mbox{ odd}
		\end{array}
		\right.
	$$
	e.g. $X$ consists of two $S$-orbits: the even numbers and the odd numbers. Since $X = \Z$ is countable we have that the countable-fold coproduct of subsets of $X$ is also countable.  Consequently, the type of a given set $P$ is particularly simple, returning the sum of types: 
	$$[P] = (\mbox{\# of evens in } P) \cdot [\{0\}] + (\mbox{\# of odds in } P) \cdot [\{1\}]$$
	Whence, we see that this space's type monoid $\T$ is isomorphic to $\overline{\N}^2 = \{ (x,y): x,y \in \N \sqcup \{ \infty \} \}$. As a category with arrows encoding the order relation:
	$$
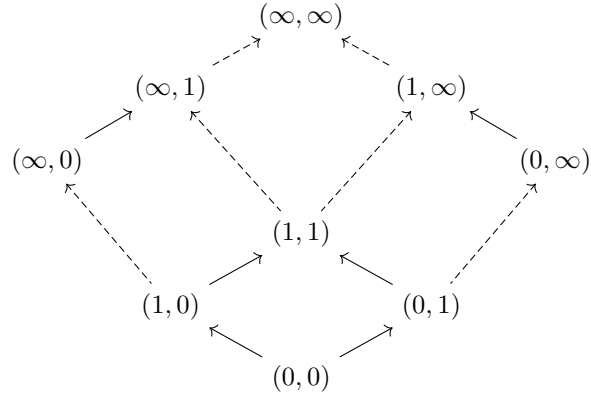
\begin{tikzcd}[column sep = small, row sep = small]
	&&   (\infty, \infty)	\ar[dashleftarrow]{dl}
							\ar[dashleftarrow]{dr}
	\\
	&    (\infty,1)			\ar[leftarrow]{dl}
							\ar[dashleftarrow]{ddr}
	&&   (1, \infty)		\ar[leftarrow]{dr}
							\ar[dashleftarrow]{ddl}
	\\
	     (\infty, 0)		\ar[dashleftarrow]{ddr}
	&&&& (0,\infty)			\ar[dashleftarrow]{ddl}
	\\
	&&   (1,1)				\ar[leftarrow]{dl}
							\ar[leftarrow]{dr}
	\\
	&    (1,0)				\ar[leftarrow]{dr}
	&&   (0,1)				\ar[leftarrow]{dl}
	\\
	&&   (0,0)
	\\
	\end{tikzcd}$$
	And associated to $\T$ we have the distributive lattice of scales given by:
	$$\begin{array}{ccc}
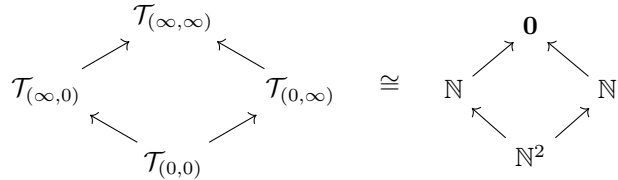
	\begin{tikzcd}[column sep = small, row sep = small]
		&  \T_{(\infty, \infty)}	\ar[leftarrow]{dl}
									\ar[leftarrow]{dr}
		\\
		   \T_{(\infty, 0)}			\ar[leftarrow]{dr}
		&& \T_{(0, \infty)}			\ar[leftarrow]{dl}
		\\
		&  \T_{(0, 0)}
	\end{tikzcd}
	& \cong &
	\begin{tikzcd}[column sep = small, row sep = small]
		&  \textbf{0}	\ar[leftarrow]{dl}
						\ar[leftarrow]{dr}
		\\
		   \N			\ar[leftarrow]{dr}
		&& \N			\ar[leftarrow]{dl}
		\\
		&  \N^2
	\end{tikzcd}
	\end{array}$$
\end{ex}

It is good to keep an example such as this in mind when thinking about the type monoid.  Already in this simple context we see that the partial order should not be expected to be a total order. Additionally, we gain intuition about the role of idempotents in addition as partially absorbing elements. Moreover, it is clear that type monoids do not embed into groups though the individual scales appear to. The following Section carries out the decomposition in the example above in general for any type monoid.

\subsection{Embedding the Type Monoid into an Inverse Semigroup}\label{embedtypemon}

An interesting takeaway from Lemma \ref{idembtw} is that we should not expect the collection of non-cancellative elements of a type monoid to coincide with its idempotents and thus emphatically, type monoids do not embed into groups.  However, $\T$ does decompose into a collection of disjoint cancellative monoids.

\begin{defn}
	The \emph{up set} of an element $a \in \T$ is the collection of elements greater than or equal to $a$;
	$$ \uparrow_a \,:= \{ b \in \T: a \preceq b \}$$
	For $e \in \E(\T)$ we define the \emph{isotropy monoid} of $e$:
	\begin{equation}\begin{array}{lcl}
		  \T_e
		& := 
		& \uparrow_e \setminus \Big( \bigcup \{ \uparrow_f : e \preceq f \in \E(\T) \} \Big)
	\end{array}\end{equation}
\end{defn}

\begin{prop}\label{isomon}
	$\T_e$ is a cancellative monoid with unit $e$. Moreover, every $a \in \T$ belongs to exactly one $\T_e$.
\end{prop}
\begin{proof}
	First observe that $\T_e$ is a monoid as for $a,b \in \T_e$ we have:
	\begin{enumerate}[(i)]
		\item $a+e = a$ by Corollary~\ref{shortcuts}
		\item $a+b \in \T_e$ since for $f \in \E(\T)$ we have $a,b \preceq f \implies a+b \preceq f$
	\end{enumerate}
	Now let's establish cancellativity of the isotropy monoid $\T_e$.  Choose a minimal noncancellative element $c \in \T_e$.  By definition;
	$$ \exists a \neq b \in \T_e \mbox{ such that } a+c = b+c $$
	Suppose for some $x \in \T_e$, $x \preceq c$. Then there is a $y \in \T$ such that $x+y = c$.  Because $c$ is not cancellative it can not be that $y \in \T_e$ for otherwise by minimality of $c$,
	$$ (a+x)+y = (b+x)+y \implies a+x = b+x \implies a = b $$
	which is a contradiction.  Thus, $y \preceq e$ and consequently $x = x+y = c$.  Therefore, $c$ is minimal in $\T_e$ e.g. $c = e$.  But $e$ being the identity, is cancellative in $\T_e$.
	Finally, note that $\T = \coprod\limits_{e \in \E(\T)} \T_e$ since for $a \in \T \implies a \in \T_e$ for $e := \max\{ e \in \E(\T): e \preceq a \}$.
\end{proof}

\begin{defn}
	Given $e \in \E(\T)$ define $\Q_e$ the \emph{isotropy group of }$e$, to be the Grothendieck group of $\T_e$.
\end{defn}
Because the isotropy monoids $\T_e$ are commutative and cancellative, $\T_e$ embeds (as the positive cone) into the abelian group $\Q_e$ which is thus congruent to $\T_e \cup -\T_e$.  The Embedding theorem follows easily from this observation. We define the \emph{quantity space of} $(X,S,\A)$ to be $\Q := \bigsqcup\limits_{e \in \E(\T)} \Q_e$.  With this, we are now able to present the proof of The Embedding Theorem.

\begin{proof}[Proof of Theorem \ref{emtheo}]
	By Proposition~\ref{isomon}, $\T$ is a distributive lattice of cancellative monoids.\footnote{Actually, the theorem follows from this observation by Theorem 7.52 of Clifford and Preston's text \emph{The Algebraic Theory of Semigroups Vol. II}~\cite{clifford1967}.} Individually, each $\T_e$ embeds into $\Q_e$. Together this collection $\Q = \bigsqcup \Q_e$ is naturally an inverse monoid under the following addition operation.  Given $a \in \Q_e$ and $b \in Q_f$,
	\begin{equation}
		a+b = (a+e) + (b+f) = (a+f) + (b+e) \in \Q_{(e \vee f)}
	\end{equation}
	That this operation is commutative, and extends the addition operation of $\T$ is obvious.  As $(\Q,+)$ is commutative and possesses weak inverses: namely, $-a$ for any $a \in \Q$, we have that $\Q$ is an inverse monoid.
\end{proof}

\subsection{Limits in the Type Monoid}
Limits are ubiquitous.  Section~\ref{catlimits} gives some examples of categorical limits. Within the context of type monoids we will chiefly be concerned with $\omega$\emph{-directed colimits}.  These are the limits of increasing sequences.  They capture inductive approximation.  The purpose of this Section is to prove that inductive approximation by measurable sets commutes with measurement e.g. the type of the limit is the limit of the types of the sets in the sequence. This notion allows us to prove continuity from above and below as given by Theorem~\ref{basics} in their most general forms.

In $\overline{\A}$ as in $\Set$, a colimit over the ordinal $\omega = \{ 0 \to 1 \to 2 \to \cdots \}$ or an $\omega$\emph{-directed colimit} specifies a nested sequence of (measurable) sets and a limit, $\varinjlim A_n = \bigcup A_n$ together with inclusion maps from each of the $A_i$.  That this union is a categorical limit is a consequence of it being the `smallest' object into which ALL of the $A_i$ include.  This is a universal property and is made precise by the following:

Given a functor $F: \omega \to \overline{\A}$ a colimit of $F$ consists of an object $\varinjlim F =: \varinjlim A_n \in \overline{\A}$ and an inductive cone (of inclusion maps) $\iota_i: F(i)=:A_i \to \varinjlim F$ which is universal in the sense that for any other object $\Theta \in \overline{\A}$ such that a commutative diagram of this type (called a cone) exists
\begin{equation}\begin{tikzcd}[column sep = small]
		  \cdots A_i		\ar{rr}[description]{\cdots}{f_{ik}}
							\ar{dr}[swap]{\theta_i}
		&
		& A_k \cdots		\ar{dl}{\theta_k}
		\\
		& \Theta
\end{tikzcd}\end{equation}
we have a unique map $u: \varinjlim A_n \to \Theta$ such that $\theta_i = u \circ \iota_i$ for all $i$.\footnote{For the detailed exposition of this concept see Appendix~\ref{cats}} In diagrams, it is the case that;
\begin{equation}\begin{tikzcd}[column sep = small]
		  \cdots A_i			\ar{rr}[description]{\cdots}{f_{ik}}
								\ar{dr}[swap]{\iota_i}
								\ar[bend right]{ddr}[swap]{\theta_i}
		&
		& A_k \cdots			\ar{dl}{\iota_k}
								\ar[bend left]{ddl}{\theta_k}
		\\
		& \varinjlim A_n		\ar{d}{u}
		\\
		& \Theta
\end{tikzcd}\end{equation}

As in classical measure theory, we'd like to be able to switch between using sequences of measurable sets and their associated sequences of types.  In this setting, that corresponds to the quotient $\overline{\A} \mapsto \overline{\A} \big/ \eqsim_S$, preserving $\omega$-directed colimits.
\begin{prop}\label{limcom}
	The $\omega$-directed colimit in $\overline{\A}$ commutes with the type:
	\begin{equation}
		[\varinjlim A_n] = \varinjlim [A_n]
	\end{equation}
	and hence Tarski measure.
\end{prop}

\begin{proof}
	It is clear that the quotient takes the universal cone of an $\omega$-directed colimit in $\overline{\A}$ to a cone in $\T$.  Explicitly, given an $\omega$-directed colimit $\varinjlim A_n \in \overline{\A}$ passing to equivalence classes we have for $i \leq j \leq k$:
	\begin{equation}\begin{tikzcd}
		  \cdots \left[ A_i \right]		\ar{r}{[f_{ij}]}
		  								\ar{dr}[swap]{[\iota_i]}
		& \left[A_j \right]				\ar{r}{[f_{jk}]}
										\ar{d}{[\iota_j]}
		& \left[ A_k \right] \cdots		\ar{dl}{[\iota_k]}
		\\
		& \left[ \varinjlim A_n \right]
	\end{tikzcd}\end{equation}
	Therefore, it suffices to prove that this cone is universal. Let $\Theta \in \T$ together with morphisms $\theta_n : [A_n] \to \Theta$ be another cone.  We must exhibit a unique morphism $u: \big[ \varinjlim A_n \big] \to \Theta$ such that $\theta_n = u \circ \iota_n$. That is, that $\big[ \varinjlim A_n \big] \preceq \Theta$ but we have
	$$
	\big[ \varinjlim A_n \big]
	=
	\left[ \bigcup A_n \right]
	=
	\left[ \coprod (A_n \setminus A_{n-1}) \right]
	=
	\sum \left[ (A_n \setminus  A_{n-1}) \right]
	\hspace{15pt}\mbox{ where } A_{-1} := \varnothing
	$$
	Consequently, $\sum\limits_1^n [(A_i \setminus A_{i-1})] = [A_n] \preceq \Theta$ for all $n$. And so, $\big[ \varinjlim A_n \big] \preceq \Theta$.
\end{proof}

\begin{prop}\label{limneg}
	Given $e \in \E(\T)$, $\omega^{op}$-directed limits in $\T_e$ are in one to one correspondence with $\omega$-directed colimits in $-\T_e$.  Specifically,
	\begin{equation}
		- \varprojlim [A_j] = \varinjlim \left( -[A_j] \right)
	\end{equation}
\end{prop}
\begin{proof}
	This follows immediately from the observation that $-\T_e$ viewed as a category is precisely $\T_e^{op}$ the opposite category of $\T_e$.
\end{proof}

Having observed the existence of $\omega$-directed colimits in the type monoid and their coincidence with such limits in $\overline{\A}$, we finally have the necessary tools to present a detailed proof of Theorem~\ref{basics}.

\begin{proof}[Proof of Theorem~\ref{basics}]
	The Monotonicity property of Tarski measure follows trivially from the observation that for $A \subseteq B$, $m(B) = m(B \setminus A) + m(A) \succeq m(A)$ by definition.
	
	Next, let $\{ A_j \} \subseteq \A$ be a countable collection of measurable sets. Define $B_1 := A_1$ and $B_j := A_j \setminus \left( \bigcup\limits_{i=1}^{j-1} A_i \right)$ for $j>1$.  Then the $B_j$'s are pairwise disjoint and $\coprod\limits_i B_i = \bigcup\limits_i A_i$ for all $n$.  Hence, by monotonicity:
	\begin{equation}
		m \left( \bigcup\limits_{i=1}^{\infty} A_i \right)
		=
		m \left( \coprod\limits_{i =1}^{\infty} B_i \right)
		=
		\sum\limits_{i=1}^{\infty} m(B_i) \preceq \sum\limits_{i=1}^{\infty} m(A_i).
	\end{equation}
	and thus Tarski measure is Subadditive.
	For Continuity from below let $A_1 \subseteq A_2 \subseteq A_3 \subseteq \cdots $ be an increasing sequence of sets in $\A$ and observe that together with their union $\cup A_i$, this is an $\omega$-directed colimit. By Proposition~\ref{limcom} it follows that:
	\begin{equation}
		m \left( \bigcup\limits_{i=1}^{\infty} A_i \right)
		=
		m \left( \varinjlim A_i \right)
		=
		\varinjlim m(A_i)
	\end{equation}
	
	Finally, suppose $\{ A_i \}_{i=1}^{\infty} \subseteq \A$ is a decreasing sequence of sets: $A_1 \supseteq A_2 \supseteq \cdots$ with the property that there exists $e \in \E(\T)$ and $N \in \N$ such that for all $n \geq N$, $m (A_n) \in \T_e$ the isotropy monoid of $e$.
	
	Define $B_j := A_N \setminus A_j$ for $j > N$ and observe that $B_{N+1} \subseteq B_{N+2} \subseteq \cdots$ with $m(A_N) = m(B_j) + m(A_j)$ for $j > N$ and $\bigcup\limits_{N+1}^\infty B_j = A_N \setminus \left( \bigcap\limits_1^\infty A_j \right)$.
	By continuity from below, then
	\begin{equation}\begin{array}{lclcl}
		  m(A_N)
		& =
		& m \left( \bigcap\limits_1^\infty A_j \right) + \varinjlim m(B_j)
		\\
		  
		& =
		& m \left( \bigcap\limits_1^\infty A_j \right) + \varinjlim m \left(  A_N \setminus A_j \right)
		\\
		  
		& =
		& m \left( \bigcap\limits_1^\infty A_j \right) + \varinjlim \Big( m(A_N) - m(A_j) \Big)
		& =
		& m \left( \bigcap\limits_1^\infty A_j \right) + m(A_N) - \varprojlim m(A_j)
	\end{array}\end{equation}
	by Proposition~\ref{limneg}.  By assumption $m(A_n) \in \T_e$ for all $n \geq N$ consequently operating in the isotropy group $\Q_e$, we may subtract $\big( m(A_N) - \varprojlim m(A_j) \big)$ from both sides of the equation to yield:
	\begin{equation}
		\varprojlim m(A_j) = e + \varprojlim m(A_j) = (m(A_N) - m(A_N)) + \varprojlim m(A_j) = e + m \left( \bigcap\limits_1^\infty A_j \right)
	\end{equation}
	and therefore Tarski measure has the Continuity from above property.\footnote{This is an adaptation of the proof in Folland~\cite{folland1984}}
\end{proof}

\section{Hierarchical Measures \& Minimal $\sigma$-ideals}\label{hierarchical}
Having established the basic properties of Tarski measure, we now turn our attention to a broader perspective in analysing stationary measures taking values in commutative monoids. To that end, we introduce the following definitions.

\begin{defn}
	Given a stationarily measurable space $(X,S,\A)$ and a commutative monoid $(T,+,0)$; a map $\nu: \A \rightarrow T$ is a \emph{stationary T-measure} if;
		\begin{enumerate}
			\item $\nu(\varnothing) = 0$
			
			\item $\nu\left( \bigcup\limits_{i = 1}^{\infty} A_i \right) = \sum\limits_{i=1}^{\infty} \nu(A_{i})$ for any sequence of pairwise disjoint sets $\{A_{i}\}^{\infty}_{i=1} \subseteq \A$.
			
			\item $\nu(A) = \nu(s^{-1}A)$
		\end{enumerate}
	A stationary $T$-measure is said to be \emph{monotone} if the group of units of its image $\nu(\A)$ is trivial. If a stationary $T$-measure has the property that all idempotents in its image are extremal elements then the measure is said to be \emph{aparadoxical}.
\end{defn}

For each $e \in \E(\T)$ there is a canonical aparadoxical monotone stationary measure taking values in the \emph{completion} of the isotropy monoid defined $\overline{\T_e} := \T_e \cup \{ \mbox{minimal upper bounds in } \E(\T) \}$ defined by the map:
\begin{equation}
	m_e(A) :=
	\left\{ \begin{array}{ll}
			[A]+e & : \mbox{ if } [A]+e \in \T_e
			 \\
			\max\{f \in \overline{\T_e} \setminus \T_e: f \preceq [A]+e \} & : \mbox{ else}
	\end{array} \right.
\end{equation}
And associated to this measure we have a $\sigma$-ideal of null sets given by:
\begin{equation}
	\mathcal{N}_e := \{ N \in \A: m_e(A) = e \}
\end{equation}
These measures are called the \emph{hierarchical measures}.  Their $\sigma$-ideals will be referred to as the \emph{minimal $\sigma$-ideals}.  The reason for this terminology is that a given an aparadoxical monotone stationary $T$-measure $\nu$ on $(X,S,\A)$, defines a $\sigma$-homomorphism from $\overline{\T_e} \to T \sqcup \{ \infty \}$ for some unique $e \in \E(\T)$ where $T \sqcup \{ \infty \}$ is $T$ with an adjoined absorbing element $\infty$ provided $T$ does not already have an absorbing element. As a consequence of this we get that the $\sigma$-ideal of $\nu$-null sets must contain $\mathcal{N}_e$.

\begin{prop}\label{extmeas}
	If $\nu$ is an aparadoxical monotone stationary $T$-measure on $(X,S,\A)$ then $\exists !$ monoid morphism $\overline{\nu} : \overline{\T_e} \to T \sqcup \{ \infty \}$ such that $\nu = \overline{\nu} \circ m_e$.
\end{prop}
\begin{proof}
	Assume that $\nu$ is as in the hypothesis. Since measures are invariant with respect to equidecomposition, $\nu$ is well-defined on types representable in $\A$ and we may define $\overline{\nu}([A]) := \nu(A)$ for $[A]$ such that there exists representative $[A] = A' \in \A$.  Let $e$ be the maximum idempotent such that $\overline{\nu}(e)$ is defined and equal to $0$ the identity in $T$.  By Lemma~\ref{rep}, every $A \in \overline{\A}$ is representable as the countable coproduct of sets in $\A$.  Using this, we extend the definition of $\overline{\nu}$ as follows.  Let $[A] = \left[ \coprod A_j \right] \in \T$ such that $A_j \in \A$. Then define
	$$
	\overline{\nu}([A]) := 
	\left\{ 
	\begin{array}{cl}
		\sum \nu(A_j)	& : \mbox{ provided the sum is defined in } T
		\\
		\infty			& : \mbox{ else}
	\end{array}
	\right.
	$$
	By Corollary~\ref{idemsum2}, every idempotent in $\T$ is equal to the countable-fold sum of itself. So because $\nu$ is aparadoxical monotone we have that for any idempotent $f \in \T$, $\overline{\nu}(f) = \sum \overline{\nu}(f) \in \{ 0, \infty \}$.  In particular, $\overline{\nu}: \T_e \to T \sqcup \{ \infty \}$ and restricted to $\A$, we have $\nu = \overline{\nu} \circ m_e$ as desired.
\end{proof}

Theorem~\ref{meascoor} is a direct consequence of Proposition~\ref{extmeas}.  One need only specialize its statement to $T \sqcup \{\infty\} = \overline{\R}^+$ and observe that since $\overline{\R}^+$ contains only extremal idempotents $\{ 0, \infty \}$ and possesses only one unit $\{ 0 \}$; ANY stationary $\overline{\R}^+$-measure is in fact aparadoxical and monotone.

\section{Closing}
We've demonstrated that Tarski's notion of \emph{type} provides an intrinsic measure on any stationarily measurable space. These measures satisfy the basic properties of classical measures in very wide generality.  However, this paper only covers the construction and most basic features of Tarski measure.

Future research directions include studying the representation theory of quantity spaces, the development of a generalized 'signed' theory, integration, differentiation, etc.  It remains to be seen precisely how much of the classical Analysis can be repeated in this setting. Because there are no non-trivial null sets with respect to Tarski measure and hence, no singular functions; we speculate that it should be possible to prove a Fundamental Theorem of Calculus for arbitrary continuous (or even measurable) maps.

Additionally, since type monoids embed into inverse semigroups, there is an endofunctor on $\SM$ given by mapping a stationarily measurable space $(X,S,\A) =: X \mapsto \left( \Q(X), \Q(X), \B_{\Q(X)} \right)$ the quantity space, acting on itself equipped with the Borel $\sigma$-algebra of its specialization topology. The features of this endofunctor are unknown and warrant future study. For example if the isotropy groups are necessarily locally compact with respect to the interval topology, then Quantity spaces would have Haar measures at every scale.  What this would mean for the quantity space of a quantity space is unknown.

We would eventually like to see applications of this work to problems in probability, physics, and dynamical systems.  The possibility of types transitioning between isotropy monoids under iterated mapping is particularly enticing as it may provide a measure-theoretic model for the transition between laminar and turbulent flow. Loosely speaking, we are interested in questions concerning the interactions between scales in systems and suspect that Tarski measure can be a useful formalism in such settings.

\section{Acknowledgements}
I'd like to thank Doctor Stephen Bricher of Linfield College for first fostering my lively distrust in and respect for Measure Theory; Tom LaGatta for his continuing enthusiasm and support throughout the writing of this paper; Joey Hirsch, John Terilla and Yash Jhaveri for helping me clarify my ideas these last several months, the inimitable Chuggs; my friends; my family; and as ever, the men and women of mathematics who laid the foundations upon which this work is built.  Thank you all.

\appendix

\section{Symmetry: Inverse Semigroups}\label{sym}

\subsection{Background}
Human preoccupation with symmetry goes back as far as anyone cares to look.  Only in about the last three centuries have mathematicians attempted to codify it axiomatically. Notable among these endeavours are Lagrange and Galois' work on solvability of polynomial equations, Klein's Erlangen Program~\cite{klein2008} which in large part is responsible for the dictum, ``symmetries ARE groups," and later the work of Veblen and Whitehead in differential geometry generalizing Lie's \emph{pseudogroups}~\cite{lawson1998}.  From there, the centipedal tree of generalizations grows ever more subtly.\footnote{See for example, Ronnie Brown's work on groupoids~\cite{brown1987}, Alex Weinstein's excellent survey of examples~\cite{weinstein1996}, Hollings' work on the Ehresmann-Schein-Namboripad Theorem~\cite{hollings2012}, etc.} Because the type construction naturally extends to an endofunctor on $\SM$ when dealing with inverse monoid actions but not when dealing with group actions, it makes sense to focus on the particular generalization given by the theory of \emph{inverse semigroups} or equivalently of \emph{inductive groupoids}.

Inverse Semigroup Theory developed initially in an effort to generalize the Erlangen Program for application to differential geometry. It was pointed out by Veblen and Whitehead~\cite{veblen1932} that even long before the Erlangen Program was formulated, there were known to be Riemannian geometries whose groups of automorphisms are trivial.  And consequently, already by 1932 there was:
\begin{quotation}
...therefore, a strong tendency among contemporary geometers to seek a generalization of the Erlanger Programm which can replace it as a definition of geometry by means of the group concept.
\end{quotation}
This marked the call for an axiomization of the pseudogroups which naturally arose in differential geometry and general relativity.
\begin{defn}
	A \emph{pseudogroup} $\Gamma$ is a collection of partial homeomorphisms between open subsets of a topological space such that $\Gamma$ is closed under composition and inverses, where we compose $\alpha, \beta \in \Gamma$ only if $\Im(\alpha) = \Dom(\beta)$.
\end{defn}
Following this definition, via partial composition of partial one-to-one transformations of a space leads one to the concept of an \emph{inductive groupoid} as in Ehresmann's work.  In another direction, Wagner and Preston 'completed' this partial composition to an always defined composition by realizing it as a special case of the composition of binary relations and adjoining an empty transformation. In so doing, they introduced \emph{inverse semigroups}.

Insofar as they were each responses to the same problem: that of finding a suitable algebraic framework for the study of systems of partial one-to-one transformations, it is perhaps unsurprising that an equivalence of categories exists between categories of inverse semigroups and of inductive groupoids.  However, it was not until the works of Schein and Nambooripad in the 1970's that the requisite tools to prove such a theorem were available.  Indeed, even then it is only recently in Lawson's 1998 text~\cite{lawson1998} that this theorem (referred to as Ehresmann-Schein-Nambooripad) was first explicitly formulated.

We now turn our attention to presenting the basic definitions and properties of inverse semigroups and presenting, at the close of this Appendix, an example drawn from Lawson~\cite{lawson1998} relating inverse semigroups to \emph{local structures}.

\subsection{Inverse Semigroups}\footnote{We draw extensively from Lawson~\cite{lawson1998} for content on semigroups.  Additional information from Howie~\cite{howie1976} is sometimes used.}
As previously mentioned inverse semigroups were introduced by Wagner and Preston to abstractly identify the structures represented by pseudogroups.  The original definition follows:
	\begin{defn}
		An \emph{inverse semigroup} $(S, \cdot)$ is a semigroup satisfying the conditions:
		\begin{description}
			\item[S is regular:] For every $s \in S$ there is an element $t$ called a \emph{weak inverse} of $s$, such that: $s = sts$ and $t = tst$.
			\item[Idempotents Commute:] If $e, f \in S$ are idempotent then $ef = fe$.
		\end{description}
	\end{defn}
The modern definition soon followed: An inverse semigroup is a regular semigroup with unique weak inverses.
	\begin{thm}
		Suppose $S$ is a regular semigroup.  Then the idempotents of $S$ commute $\iff$ every element $s \in S$ has a unique weak inverse denoted $s^*$.
	\end{thm}
	
Under this definition a pseudogroup may be described by an inverse semigroup of partial homeomorphisms of open subsets of a topological space $X$ where we take the composition of $\alpha: \Dom(\alpha) \to X$ and $\beta: \Dom(\beta) \to X$ to be the function $\beta \circ \alpha: \alpha^{-1}(\Dom \beta \cap \Im \alpha) \to X$.  In this manner the idempotents of the pseudogroup are the restriction of the identity map to open subsets of $X$ and the weak inverses are given by inverses on the image of a given map.
	
	Summarized below are the basic properties of the weak inverses.
	
\begin{prop}
	Let $(S, \cdot)$ be an inverse semigroup.
	\begin{enumerate}[(i)]
		\item For any $s \in S$, both $s^* s$ and $s s^*$ are idempotents and $s (s^* s) = s$ and $s (s^* s) = s$.
		
		\item $(s^*)^* = s$ for every $s \in S$.
		
		\item For any idempotent $e \in S$ and any $s \in S$, the element $s^* e s$ is an idempotent.
		
		\item If $e$ is an idempotent of $S$ then $e^* = e$.
		
		\item For all $s_1, s_2,..., s_n \in S$, $(s_1 s_2 \cdots s_n)^* = s^*_n s^*_{n-1} \cdots s^*_1$.
	\end{enumerate}
\end{prop}
	
We take up the common notation for the collection of idempotent elements of a semigroup $S$, defining $\E(S) := \{ e \in S: e^2 = e \}$.  The idempotents $\d_s := s^*s$ and $\r_s := ss^*$ are of special significance and will be referred to respectively as the \emph{domain idempotent} and \emph{codomain idempotent} of $s$.
	
Additionally, inverse semigroups possess a natural partial order given by:
\begin{defn}
	Given an inverse semigroup $(S,\cdot)$ the natural partial order is defined
	$$
		s \leq t \iff s = te \hspace{25pt}\mbox{for some } e \in \E(\T)
	$$
\end{defn}
\noindent The basic results on the partial order are summarized in the following two propositions:
\begin{prop}
	Let $(S,\cdot)$ be an inverse semigroup.  Then the following are equivalent
	\begin{enumerate}[(i)]
		\item $s \leq t$
		\item $s = ft$ for some idempotent $f$
		\item $s^* \leq t^*$
		\item $s = ss^*t$
		\item $s = ts^*s$
	\end{enumerate}
\end{prop}
Two of these properties deserve special comment.  Firstly, the asymmetry in the definition of the partial order--that the idempotent appears on the right--is superficial by property (ii).  Secondly, property (iii) demonstrates that weak inversion does not change the order relation as one might expect.
\begin{prop}
	Let $(S,\cdot)$ be an inverse semigroup.
	\begin{enumerate}[(i)]
		\item The relation '$\leq$' is a partial order on $S$.
		\item For idempotents $e,f \in S$ we have that $e \leq f \iff e = ef = fe$.
		\item If $s \leq t$ and $u \leq v$ then $su \leq tv$.
	\end{enumerate}
\end{prop}

\begin{defn}
	Let $(P,\leq)$ be a poset.  If $z \leq x,y$ then $z$ is a \emph{lower bound} of $x$ and $y$.  If $z$ is the maximum of the lower bounds it is called the \emph{greatest lower bound} and denoted by $x \wedge y$.  A \emph{meet semilattice} is a poset in which every pair of elements has a greatest lower bound. 
\end{defn}

\begin{prop}
	The set of idempotents $\E(S)$ of an inverse semigroup $S$ together with its induced partial order by $\leq$, form a meet semilattice. 
\end{prop}
	
In light of the previous proposition, it common to refer to $\E(S)$ as the \emph{semilattice of idempotents}.  Indeed, it is the case that every meet semilattice is an inverse semigroup in which all members are idempotent.
	
Wagner and Preston proved the inverse semigroup analogue to Cayley's Theorem in group theory, demonstrating that \emph{every} inverse semigroup is realized as an inverse semigroup of partial bijections of \emph{some} set.
\begin{defn}
	The \emph{symmetric inverse monoid} $I(X)$ on a set $X$ is defined to be the inverse monoid of partial bijections $f: A \to A'$ for some $A,A' \subseteq X$.
\end{defn}
To clarify, a \emph{partial bijection} as above is a bijective function whose domain and codomain are subsets of a given set.  For example, the collection of idempotent partial bijections of a given set $X$ is comprised exactly of the identity maps, $\{ id_A: A \to A: A \subseteq X \}$.
\begin{thm}[\textbf{Wagner-Preston Representation Theorem:}]
	Let $S$ be an inverse semigroup.  Then there is a set $X$ and an injective homomorphism $\Phi: S \to I(X)$ such that
	\begin{equation}
		a \leq b \iff \Phi(a) \subseteq \Phi(b)
	\end{equation}
\end{thm} 

As a standard choice of a set $X$ supporting a representation of a given inverse semigroup we may choose $X = S$ and map the member $s \in S$ to $\rho_s \in I(S)$ defining $\rho_s$ to be the partial bijection with domain $\d(\rho_s) = Ss^* = \{ ts^* : t \in S \}$ with $\rho_s(t) := ts$, for all $t \in \d(\rho_s)$.

\subsection{Local Structures}
The following excerpt is taken directly from Lawson's text \emph{Inverse Semigroups}. Proofs have been omitted. Readers interested in a detailed exposition are directed to Lawson~\cite{lawson1998}. Presently, our purpose is merely to indicate the natural occurrence of geometries represented by Inverse Semigroups and their actions on spaces.

At its simplest, differential geometry concerns spaces which look locally like pieces of $\R^n$ and the pseudogroups provide the glue to hold these pieces together.  We now introduce the mathematics which will enable us to make this idea precise.

Let $X$ and $Y$ be topological spaces.  In what follows, the space $X$ will be the \emph{model space} and our aim will be to construct geometric structures on $Y$ which are locally like pieces of $X$.  A \emph{chart} from $X$ to $Y$ is a homeomorphism $\phi: U \to V$ between open subsets of $X$ and $Y$ respectively.  Thus a chart is nothing other than a special kind of partial bijection. An \emph{atlas} from $X$ to $Y$ is a collection of charts from $X$ to $Y$ such that the union of the images of the charts is $Y$; thus $1_Y = \bigcup\limits_{\phi \in \A} \phi \phi^{-1}$. If this latter condition does not hold, we shall sat that $\A$ is a \emph{partial atlas}.

If $\A$ is a partial atlas from $X$ to $Y$ and $\B$ is a partial atlas from $Y$ to $Z$, then $\B \A$ is the partial atlas from $X$ to $Z$ defined by
$$
\B \A = \{ \psi \phi: \psi \in \B \mbox{ and } \phi \in \A \}
$$
If $\A$ is a partial atlas from $X$ to $Y$ then
$$
\A^{-1} = \{ \phi^{-1}: \phi \in \A \}
$$
is a partial atlas from $Y$ to $X$.

Intuitively, the existence of an atlas from $X$ to $Y$ means that $Y$ can be described by a family of overlapping sets each of which looks like a piece of $X$.  An atlas, in the geographical sense, provides a good example of an atlas in our sense from $\R^2$ to a sphere.

The problem now arises of dealing with the overlaps between different charts, and it is here that pseudogroups get into the picture.

Let $\phi_i: U_i \to V_i$ and $\phi_j: U_j \to V_j$ be any charts in the (partial) atlas $\A$ from $X$ to $Y$.  Then we can form the partial homeomorphism
$$
\phi_j^{-1} \phi_i: \phi_i^{-1}(V_i \cap V_j) \to \phi_j^{-1}(V_i \cap V_j)
$$
between the open subsets $\phi_i^{-1}(V_i \cap V_j)$ and $\phi_j^{-1}(V_i \cap V_j)$ of $X$.  Partial homeomorphisms such as these are called \emph{transition functions} of the atlas.  Thus the transition functions of a (partial) atlas $\A$ from $X$ to $Y$ belong to the pseudogroup $\Gamma(X)$ of all partial homeomorphisms between the open subsets of $X$.  This property can be succinctly expressed by the equation $\A^{-1} \A \subseteq \Gamma(X)$.
In order to construct geometric structures, we shall need to know under what circumstances partial bijections can be glued together.

\begin{prop}
	Let $f,g: X \to Y$ be partial bijections between the sets $X$ and $Y$.
	\begin{enumerate}[(i)]
		\item $f \cup g$ is a partial bijection precisely when $fg^{-1}$ is an idempotent.
		\item $f \cup g$ is a partial bijection precisely when $f^{-1}g$ and $fg^{-1}$ are idempotents.
		\item Let $\{ \phi_i : i \in I \}$ be a family of partial bijections from $X$ to $Y$.  Then $\cup \phi_i$ is a partial bijection $\iff$ $\phi_i \cup \phi_j$ is a partial bijection for all $i,j \in I$.
		\item Let $\{ \phi_i : i \in I \}$ be a family of partial bijections from $X$ to $Y$ such that $\cup \phi_i$ is a partial bijection.  Then $\cup phi_i^{-1}$ is a partial bijection from $Y$ to $X$ and
		$$
			\left( \cup \phi_i \right)^{-1} = \cup \phi_i^{-1}.
		$$
		\item Let $\{ \phi_i : i \in I \}$ be a family of partial bijections from $X$ to $Y$ such that $\cup phi_i$ is a partial bijection. Then from any partial bijection $\phi: W \to Y$ the union $\cup \phi_i \phi$ is a partial bijection from $W$ to $Y$ such that
		$$
			\left( \cup \phi_i \right) \phi = \cup \phi_i \phi.
		$$
		Similarly, for any partial bijection $\psi: Y \to Z$ the union $\cup \psi \phi_i$ is a partial bijection from $X$ to $Z$ such that 
		$$
			\psi \left( \cup \phi_i \right) = \cup \psi \phi_i.
		$$
	\end{enumerate}
\end{prop}

On the basis of (ii) above, partial bijections $f$ and $g$ are said to be \emph{compatible}, written $f \sim g$ if $f \cup g$ is a partial bijection.  An inverse subsemigroup $S$ of $I(X)$ is said to be \emph{complete} if the union of every non-empty compatible subset of $S$ belongs to $S$.  Clearly, $I(S)$ is complete.

Let $X$ be a topological space and $\Gamma = \Gamma(X)$ be the pseudogroup of all partial homeomorphisms between open subsets of $X$.  A \emph{complete pseudogroup} on $X$ is an inverse subsemigroup $\Gamma' \subseteq \Gamma$ with the property that the union of every non-empty compatible subset of $\Gamma'$ belongs to $\Gamma'$. 

Let $X$ and $Y$ be topological spaces, where $X$ is the model space, and let $\Gamma'$ be a complete pseudogroup on $X$.  An atlas $\A$ from $X$ to $Y$ is said to be \emph{compatible} with $\Gamma'$ if $\A^{-1}\A \subseteq \Gamma'$.  We are therefore requiring that all the transition functions belong to a specified inverse subsemigroup of $\Gamma(X)$.

Let $\A$ and $\B$ be atlases from $X$ to $Y$ compatible with $\Gamma'$.  If $\A^{-1}\B \subseteq \Gamma'$ then we say that $\A$ and $\B$ are \emph{compatible modulo } $\Gamma'$, denoted $\A \sim \B \mod \Gamma'$.

\begin{prop}
	Let $\Gamma'(X,Y)$ be the set of all atlases from $X$ to $Y$ compatible with a the complete pseudogroup $\Gamma'$.
	\begin{enumerate}[(i)]
		\item The relation $\A \sim \B \mod \Gamma'$ is an equivalence relation on $\Gamma'(X,Y)$.
		\item Each equivalence class contains a maximum element.
	\end{enumerate}
\end{prop}

An atlas $\A$ from $X$ to $Y$ which is the maximum element in its equivalence class modulo $\Gamma'$ is called a \emph{complete atlas compatible with } $\Gamma'$, and is said to define a $\Gamma'$-\emph{structure} on $Y$.  Such structures are often called \emph{local structures}.  Any atlas $\A$ from $X$ to $Y$ determines a $\Gamma'$ structure on $Y$, namely the one determined by the maximum element of the equivalence class containing $\A$.  Local structures are the analogues of of the geometries in the Erlanger Programm and the pseudogroup $\Gamma'$ replaces the group.

Differential manifolds are an example of a local structure.  Let $Y$ be a Hausdorff space with a countable basis of open sets, let $X = \R^n$ and the $\Gamma'$ be the inverse semigroup of all diffeomorphisms between open sets of $\R^n$.  This is a complete pseudogroup.  Then an atlas from $\R^n$ to $Y$ compatible with $\Gamma'$ defines a \emph{smooth differentiable manifold} on $Y$

\section{Some Categorical Concepts}\label{cats}
Category Theory distinguishes itself from other foundational mathematics by being in the first place, concerned with the relationships between objects.  This directly contrasts with Set Theory and its central focus on objects: namely, sets.  Proponents laud its far reaching results and vast applicability. Detractors often cite this generality as a weakness.  Regardless, the limited category theoretic tools I use within this paper serve to simplify the expression of certain ideas.  Consequently, I value them.  Below, I present a short listing of definitions and concepts punctuated with examples drawn largely from familiar categories such as:
\begin{description}
\item[\catname{Grp}:] the category of groups together with group homomorphisms.
\item[\catname{Ab}:] the category of abelian groups together with group homomorphisms.
\item[\Top:] the topological spaces with morphisms given by continuous maps.
\item[\Set:] the category of sets with functions.
\item[\catname{P}:] where $\catname{P} := (P,\leq)$ is a poset. These are naturally categories with objects given by the members of the underlying set and morphisms given by the partial order.
\item[\Meas:] the category of measurable spaces $(X,\A)$ together with measurable maps.
\end{description}

Category Theory is an established branch of mathematics with all the attendant depth of abstraction and specialization.  However, the definitions and concepts used here are basic and basically presented.  Readers interested in becoming acquainted with Category Theory proper would do well to look through Mac Lane's text \emph{Categories for the Working Mathematician}~\cite{maclane1998}. Additionally, the nLab is an excellent resource for content and has proven time and again to be an invaluable source of discussion on category theoretic perspective and style.\footnote{\url{http://ncatlab.org}}

\subsection{Basic Definitions}
At its most basic, a \emph{category} $\catC$ consists of a collection of \emph{objects}: $Obj(\catC)$ and a collection of \emph{morphisms} or \emph{arrows}: $mor(\catC)$ together with functions:
\begin{itemize}
	\item $d,c: mor(\catC) \to obj(\catC)$ which assign, to every morphism, its \emph{domain} and \emph{codomain} respectively
	\item a partial function $\circ: mor(\catC) \times mor(\catC) \to mor(\catC)$ which assigns the \emph{composite} morphism $g \circ f$ to any pair of morphisms $f,g$ such that $d(f) = c(g)$
	\item a function $id: obj(\catC) \to mor(\catC)$ which assigns to each object $x$ an \emph{identity morphism}, $id_x$.
\end{itemize}
satisfying the following properties:
\begin{itemize}
	\item $d(g \circ f) = d(f)$ and $c(g \circ f) = c(g)$
	\item $d(id_x) = x$ and $c(id_x) = x$
	\item Composition is associative: $(h \circ g) \circ f = h \circ (g \circ f)$ when such compositions are defined.
	\item If $d(f) = x$ and $c(f) = y$, then $id_y \circ f = f = f \circ id_x$.
\end{itemize}

The \emph{opposite} category $\catC^{op}$ is the category generated by formally switching the roles of the domain and codomain functions e.g. switching the direction of all the arrows in the category.  For example, if a category $\catC$ contains diagram of arrows to left in the image below, then $\catC^{op}$ contains the diagram of arrows to the right:
\begin{equation}\begin{array}{ccc}
\begin{tikzcd}
	  x			\ar{r}
	& y			\ar{r}
	& z			
	\\
	  a			
	& b			\ar[xshift = -.5ex]{u}
				\ar[xshift = +.5ex]{u}
				\ar{l}
	\\
\end{tikzcd}
& \hspace{10ex} &
\begin{tikzcd}
	  x			
	& y			\ar{l}
				\ar[xshift = -.5ex]{d}
				\ar[xshift = +.5ex]{d}
	& z			\ar{l}
	\\
	  a			\ar{r}
	& b				
	\\
\end{tikzcd}
\end{array}\end{equation}

\subsection{Functors}
A functor is what goes between categories e.g. a morphism of categories. They were first explicitly defined in Algebraic Topology but are ubiquitous in modern mathematics forming as they do bridges between seemingly different subjects.

Formally a functor $P: \catC \to \catD$ is composed of two maps; an object map $P: obj(\catC) \to obj(\catD)$ and a morphism map $P: mor(\catC) \to mor(\catD)$ sending $f: x \to y$ to $Pf: Px \to Py$ such that 
\begin{enumerate}[(i)]
	\item $P(id_x) = id_{P(x)}$ for all $x \in obj(\catC)$
	
	\item $P(g \circ f) = P(g) \circ P(f)$ for all ${f,g} \in mor(\catC)$
\end{enumerate}
Dually, a co-functor $Q: \catC \to \catD$ is composed of an object map $Q: obj(\catC) \to obj(\catD)$ and a morphism map $Q: mor(\catC) \to mor(\catD)$ such that
\begin{enumerate}[(i)]
	\item $Q(id_x) = id_{Q(x)}$ for all $x \in obj(\catC)$
	
	\item $Q(g \circ f) = Q(f) \circ Q(g)$ \emph{in the opposite direction} for all ${f,g} \in mor(\catC)$
\end{enumerate}
That is, a co-functor is a functor from the opposite category $\catC^{op} \to \catD$.
\begin{ex}
	A simple example is the power set functor $\mathcal{P}:\Set \to \Set$.  Its object function maps a set $X \mapsto \mathcal{P}X$ the power set of $X$; the morphism map assigns to each function $f: X \to Y$ the function $\mathcal{P}f: \mathcal{P}X \to \mathcal{P}Y$ which sends each subset $P \subseteq X$ to its image $fS \subseteq Y$. 
\end{ex}
\begin{ex}
	Homology provides an example in Topology. Given a topological space $X$, taking homology provides an abelian group $H_n (X)$ for each $n \in \N$.  In addition, given a continuous map $f: X \to Y$ the construction gives a group homomorphism $H_n f: H_n (X) \to H_n (Y)$ relating the homologies of the spaces $X$ and $Y$.
\end{ex}
\begin{ex}
	\emph{Forgetful functors} are also very common.  For instance, the functor $U: \catname{Grp} \to \Set$ which takes a groups $G$ and maps them to their underlying set 'forgetting' the group structure and similarly, taking homomorphisms $f: G \to H$ to the underlying map of sets.
\end{ex}

\subsection{Limits and Colimits}\label{catlimits}
Rather than give an explicit definition of limits and colimits (via universal arrows / objects out of or into certain diagrams) which is really beyond the scope of this Appendix.  We will quote the nLab's introduction to the concept and then present a few examples of special limits.\footnote{Definition source nLab Limit: \url{http://ncatlab.org/nlab/show/limit}.}
	\begin{quotation}
		In category theory a limit of a diagram $F: \catD \to \catC$ in a category $\catC$ is an object $\lim F$ of $\catC$ equipped with morphisms to the objects $F(d)$ for all $d \in \catD$, such that everything in sight commutes. Moreover, the limit $\lim F$ is the universal object with this property, i.e. the “most optimized solution” to the problem of finding such an object.
		
		The limit construction has a wealth of applications throughout category theory and mathematics in general...So in some sense the limit object $\lim F$ “subsumes” the entire diagram $F(D)$ into a single object, as far as morphisms into it are concerned. The corresponding universal object for morphisms out of the diagram is the colimit.
	\end{quotation}
The product and coproduct constructions are two familiar examples of categorical limits.
\begin{defn}
	In a category $\catC$, the \emph{product} of two objects $X$ and $Y$ (if it exists), is an object $X \times Y$ together with two morphisms called \emph{projections} $\pi_X: X \times Y \to X$ and $\pi_Y: X \times Y \to Y$ such that for any object $Z$ for which there exist morphisms $f: Z \to X$ and $g: Z \to Y$, there exists a unique $h: Z \to X \times Y$ such that $\pi_X \circ h = f$ and $\pi_Y \circ h = g$.
\end{defn}
\noindent Equivalently, a product in $\catC$ is the limit, $\varprojlim F$ over a functor from the discrete (only arrows are identities) category with two elements $F: \{ \cdot, \cdot \} \to \C$. Some examples of products in various categories are:
\begin{itemize}
	\item $G \times H \in \catname{Grp}$ defined with the pointwise group law.
	\item The product space in $\Top$ given by the cartesian product of the underlying sets equipped with the product topology e.g. the coarsest topology for which the projection maps are continuous.
	\item in $\Set$ the product is the Cartesian Product.
	\item in a poset $(P, \leq)$ the categorical product (if it exists) of two elements $a,b$ is their meet $a \wedge b$.
\end{itemize}

Coproducts are dually defined as colimits, $\varinjlim F$ over functors $F: \{ \cdot, \cdot \} \to \catC$ or equivalently,
\begin{defn}
	In $\catC$, the \emph{coproduct} of two objects $X$ and $Y$, is an object $X \coprod Y$ together with a pair of morphisms called \emph{injections} $\iota_X : X \to X \coprod Y$ and $\iota_Y: Y \to X \coprod Y$ such that for any other object $Z$ with morphisms $f: X \to Z$ and $g: Y \to Z$ there exists a unique $u: X \coprod Y \to Z$ such that $u \circ \iota_X = f$ and $u \circ \iota_Y = g$.
\end{defn}
\noindent Some examples of coproducts are:
\begin{itemize}
	\item In $\catname{Grp}$, the coproduct of two groups $G$ and $H$ is their free product $G \star H$.\footnote{Note that the underlying set of $G \star H$ is NOT the disjoint union of the underlying sets of $G$ and $H$.}
	\item The coproduct in $\Set$ is given by the disjoint union.
	\item For a poset $(P,\leq)$ the coproduct (if it exists) of two elements $a$ and $b$ is their join $a \vee b$.
\end{itemize}

In both products and coproducts we've been taking limits over the discrete category with two elements.  If the source category for a limit is instead a \emph{directed category} such as the ordinal category $\omega := \{ 0 \to 1 \to 2 \to \cdots \}$, the resulting limits are referred to as \emph{directed limits}.  Within this work, we are primarily interested in $\omega$-directed colimits and $\omega^{op}$-directed limits as these limits, taken over functors into $\Set$ represent $\varinjlim A_n = \bigcup A_n$ where $A_0 \subseteq A_1 \subseteq \cdots$ is a nested increasing sequence of sets, and $\varprojlim B_n = \bigcap B_n$ where $B_0 \supseteq B_1 \supseteq \cdots$ is a nested decreasing sequence of sets.
\begin{defn}
	A \emph{directed set} $(I,\preceq)$ is a non-empty set with a preorder satisfying the property that given $x,y \in I$ there exists $z \in I$ such that $x,y \preceq z$. A category $\catname{I}$ with objects given by the members of a directed set and morphisms $x \to y \iff x \preceq y$ is called a directed \emph{category}.
\end{defn}
So, a directed colimit is the colimit; $\varinjlim F$ where $F: I \to \catC$ is a functor whose domain category is directed. Formally,
\begin{defn}
	Let $\C$ be a category.  An \emph{inductive system} in $\C$ consists of a directed set $I$, a family $(A_i)_{i \in I}$ of objects of $\C$, and a family $(f_{ij}:A_i \to A_j)_{i \preceq j \in I}$ of morphisms, such that:
	\begin{enumerate}[(i)]
		\item $f_{ii}: A_i \to A_i$ is the identity morphism on $A_i$
		\item $f_{ik}: A_i \to A_k$ is the composite $f_{jk} \circ f_{ij}$ for any $j$ between $i$ and $k$.
	\end{enumerate}
\end{defn}
An \emph{inductive cone} of this inductive system is an object $\Psi$ and a family of \emph{inductions} $\iota_i: A_i \to \Psi$ such that $$\iota_i = \iota_j \circ f_{ij}$$  Finally, an \emph{inductive limit} or \emph{directed colimit} of the inductive system is an inductive cone $\varinjlim A_i$ (suppressing the $f$'s and $\iota$'s) which is universal in that, given any inductive cone $\Psi$, there exists a unique morphism $u: \varinjlim A_i \to \Psi$ such that $$\iota_i = u \circ \iota_i$$ (where the left hand $\iota$ is from the cone $\Psi$ and the right hand is from the limit).\footnote{Definition source nLab Directed colimit: \url{http://ncatlab.org/nlab/show/directed+colimit}.}

The primary example of interest in this paper is the $\omega$-directed colimit in $\Set$.  A functor $F: \omega \to \Set$ identifies an inductive system.  Taking $A_i := F(i)$ and $f_{ij} := F(i \to j)$ we have the following diagram in $\Set$ representing a nested sequence of sets:
\begin{equation}\begin{tikzcd}[column sep = 4ex]
	A_0			\ar{rr}{f_{01}}
	&
	& A_1		\ar{rr}{f_{12}}
	&
	& A_2		\ar{rr}{f_{23}}
	&
	& \cdots
\end{tikzcd}\end{equation}
An inductive cone of of this inductive system is represented by the commutative diagram:
\begin{equation}\begin{tikzcd}[column sep = small]
	  A_0		\ar{rr}{f_{01}}
	  			\ar{drrr}{\iota_0}
	&
	& A_1		\ar{rr}{f_{12}}
				\ar{dr}{\iota_1}
	&
	& A_2		\ar{rr}{f_{23}}
				\ar{dl}[swap]{\iota_2}
	&
	& \cdots
	\\
	&
	&
	& \Psi
	\\
\end{tikzcd}\end{equation}
And finally, we see that the directed colimit of must be $\varinjlim A_n = \bigcup A_n$ since that is the least (in the partial order $\subseteq$) set admitting inclusion maps $\iota_n: A_n \to \Psi$ for all $n$.

We similarly define the \emph{directed limit} as the limit, $\varprojlim F$ of a functor $F$ from a directed category. The formal definition via projective systems is given by:
\begin{defn}
	A \emph{projective system} in $\catC$ consists of a directed set $I$, a family $(B_i)_{i \in I}$ of objects in $\catC$, and a family $(f_{ij}: B_i \to B_j)_{i \preceq j \in I}$ of morphisms, such that:
	\begin{enumerate}[(i)]
		\item $f_{ii}$ is the identity morphism on $B_i$
		\item $f_{ik}$ is the composite $f_{jk} \circ f_{ij}$ for any $j$ between $i$ and $k$.
	\end{enumerate}
\end{defn}
A \emph{projective cone} of this projective system is an object $\Psi$ together with a family of \emph{projections} $\pi_i: \Psi \to B_i$ such that $$\pi_i = f_{ij} \circ \pi_j$$ The \emph{projective limit} or \emph{directed limit} of the inductive system is a projective cone $\varprojlim B_i$ which is universal in that, given any projective cone $\Psi$, there exists a unique morphism $u: \Psi \to \varprojlim B_i$ such that $$\pi_i = \pi_i \circ u$$ (where the left hand $\pi$ is from the cone $\Psi$ and right hand $\pi$ is from the limit).

Again, we are primarily interested in $\omega^{op} = \{ 0 \gets 1 \gets 2 \gets \cdots \}$ directed limits in $\Set$.  The diagram of a projective cone of such a limit being:
\begin{equation}\begin{tikzcd}[column sep = small]
	&
	&
	& \Psi		\ar{dlll}[swap]{\pi_0}
				\ar{dl}{\pi_1}
				\ar{dr}{\pi_2}
	\\
	  B_0
	&
	& B_1		\ar{ll}{f_{10}}
	&
	& B_2		\ar{ll}{f_{21}}
	&
	& \cdots	\ar{ll}{f_{32}}
\end{tikzcd}\end{equation}
one can see that the limit $\varprojlim B_n = \bigcap B_n$ the largest set which is a subset of all the $B_i$.

\bibliographystyle{plain}
\bibliography{biblio_tyler}

\end{document}